\documentclass{ws-sd-bs}

\usepackage{amssymb,amsmath,amsfonts,latexsym} 
\usepackage{hyperref}

\usepackage{mathtools}
\usepackage{multirow,array}
\usepackage{graphicx}
\usepackage{subfigure}
\usepackage{epstopdf}
\usepackage{float} 
\usepackage{color, xcolor}
\allowdisplaybreaks

\usepackage{bbm, comment}

\newtheorem{hyp}[theorem]{HYPOTHESIS}

\DeclareMathAlphabet{\mathpzc}{OT1}{pzc}{m}{it}
\DeclareMathAlphabet{\mathsfsl}{OT1}{cmss}{m}{sl}

\newcommand{\FH}{\mathfrak{H}}

\newcommand{\dif}{\mathrm{d}}

\newcommand{\abs}[1]{\left\vert#1\right\vert}
\newcommand{\set}[1]{\left\{#1\right\}}

\newcommand{\E}{\mathbb{E}}

\newcommand{\Rnum}{\mathbb{R}}

\newcommand{\Be}{\begin{equation}}
	\newcommand{\Ee}{\end{equation}}
\newcommand{\Bs}{\begin{split}}
	\newcommand{\Es}{\end{split}}
\newcommand{\Bes}{\begin{equation*}}
	\newcommand{\Ees}{\end{equation*}}
\newcommand{\BT}{\begin{thm}}
	\newcommand{\ET}{\end{thm}}
\newcommand{\Bp}{\begin{proof}}
	\newcommand{\Ep}{\end{proof}}
\newcommand{\BL}{\begin{lem}}
	\newcommand{\EL}{\end{lem}}
\newcommand{\BP}{\begin{proposition}}
	\newcommand{\EP}{\end{proposition}}
\newcommand{\BC}{\begin{corollary}}
	\newcommand{\EC}{\end{corollary}}
\newcommand{\BR}{\begin{remark}}
	\newcommand{\ER}{\end{remark}}
\newcommand{\BD}{\begin{defn}}
	\newcommand{\ED}{\end{defn}}
\newcommand{\BI}{\begin{itemize}}
	\newcommand{\EI}{\end{itemize}}

\newcommand\sgn{\mathrm{sgn}}

\allowdisplaybreaks

\begin{document}

\markboth{Z. Tang,  Y. Li,   H. Yang,  H. Yi and  Y. Chen }{Berry-Ess\'{e}en bound for the ME of the fOU model with discrete fixed step size}

\catchline{}{}{}{}{}

\title{Berry-Ess\'{e}en bound for the Moment Estimation of the fractional
	Ornstein–Uhlenbeck model under fixed step size discrete observations}

\author{Zheng Tang$^{\dagger}$}

\address{
	School of Statistics and Data Science, Lanzhou University of Finance and Economics,\\
	Lanzhou, 730020, Gansu, China;\\
	School of Mathematics and Finance, Chuzhou University,\\
	 Chuzhou,  239012, Anhui, China\\
	  tzheng8889@126.com}

\author{Ying Li$^{\dagger}$}

\address{
	School of Mathematics and Computational Science, Xiangtan University,\\ 
    Xiangtan, 411105, Hunan, China\\
	liying@xtu.edu.cn}

\author{Haili Yang}

\address{
	School of Big Data, Baoshan University,\\
	 Baoshan, 678000, Yunnan, China\\
     412761802@qq.com}

\author{Hua Yi}

\address{
	School of Mathematics and Statistics, Liupanshui Normal University,\\ Liupanshui, 553004, Guizhou, China\\
	 yihua@whu.edu.cn}

\author{Yong Chen\footnote{Corresponding author.}}

\address{
		School of Big Data, Baoshan University,\\
	Baoshan, 678000, Yunnan, China\\
	zhishi@pku.org.cn}

\address{$^{\dagger}$ These authors contributed equally to this work.}

\maketitle

\begin{history}
\received{(Day Month Year)}
\revised{(Day Month Year)}
\end{history}

\begin{abstract}
Let the Ornstein–Uhlenbeck process $\set{X_t,\,t\geq 0}$ driven by a fractional Brownian motion $B^H$
described by $\dif X_t=-\theta X_t \dif t+ \dif B_t^H,\, X_0=0$ with known parameter $H\in (0,\frac34)$ be observed at discrete time instants $t_k=kh, k=1,2,\dots, n $. If $\theta>0$ and if the step size $h>0$ is arbitrarily fixed, we derive Berry-Ess\'{e}en bound for the ergodic type estimator (or say the moment estimator) $\hat{\theta}_n$, i.e., the Kolmogorov distance between the distribution of $\sqrt{n}(\hat{\theta}_n-\theta)$ and its limit distribution is bounded by a constant $C_{\theta, H,h}$ times $n^{-\frac12}$  and $ n^{4H-3}$ when $H\in (0,\,\frac58]$ and $H\in (\frac58,\,\frac34)$, respectively. 
This result greatly improve the previous result in literature where $h$ is forced to go zero. Moreover, we extend the Berry-Ess\'{e}en bound to the Ornstein–Uhlenbeck model driven by a lot of Gaussian noises such as the sub-bifractional Brownian motion and others. A few ideas of the present paper come from Haress and Hu (2021), Sottinen and Viitasaari (2018), and Chen and Zhou (2021). 
\end{abstract}

\keywords{Fractional Brownian motion; Fourth moment theorem; Berry-Ess\'{e}en bound;  Fractional Ornstein–Uhlenbeck; Sub-bifractional Brownian motion; Fractional Gaussian process; Kolmogorov distance; Malliavin calculus.}

\ccode{AMS Subject Classification: 62M09,60G22,60H10,60H30}

\section{Introduction and main results} \label{sec1}
\setcounter{section}{1} \setcounter{equation}{0}

The fractional Ornstein-Uhlenbeck processes $\set{X_t:t \geq 0}$ is known as the solution of the Langevin equation
\begin{equation}\label{fou}
	\dif X_t=-\theta X_t \dif t+ \sigma\dif B_t^H, \quad t\in[0,T], \quad 
\end{equation}
where $\sigma>0,\,\theta>0$ are the unknown parameter and $B_t^H$ is the fractional Brownian motion (fBm) with Hurst index $H\in(0,1)$.  The problem of estimation of part or all parameters $(\theta, \sigma, H)$ has been intensively studied in the last decades, see for example \cite{CLZ24,CZ21} and the references therein. In the literature, it is often assume that the process $\set{X_t:t \geq 0 }$ is observed continuously or discretely with a step size $h_n$ which is forced to go zero as $n\to \infty$, see \cite{DEKN 22} for example.

Recently, the assumption of the step size  $h_n\to 0$ as $n\to \infty$ is finally removed in \cite{HH21}. In detail,  denote $\{X_{jh}:j=1,\cdots,n\}$ the discrete-time observations of the processes $\{X_t:t \geq 0\}$, sampled at equidistant time points $t_j=jh$ with fixed step size $h$ and $n$ is the sample size. 
They propose an ergodic type statistical estimator $(\hat{\theta}_n,\,\hat{H}_n,\hat{\sigma}_n)$ for all the parameter $(\theta, H, \sigma)$ and show the strong consistence and the central limit theorem. 

Let us recall the moment estimator for unknown parameter $\theta$ under discrete observations  $\{X_{jh}:j=1,\cdots,n\}$:
\begin{equation}\label{theta hat}
	\hat{\theta}_{n}=\left(\frac{1}{{H} \Gamma(2H ) n} \sum_{j=1}^n  X_{jh}^2  \right)^{-\frac{1}{2H}}.
\end{equation}
In the present paper, we aim to show, under the framework of \cite{HH21}, the rate of convergence of the estimator $\hat{\theta}_n$ in the Kolmogorov distance, which is called the Berry-Ess\'een type upper bound in literature \cite{chenkl2020}.
We point out that in \cite{HR24}, the rate of convergence for the estimator is obtained in the $p$-Wasserstein distance. 

For simplicity, we assume that $X_0=0$ and that $H,\,\sigma$ are known and $\sigma=1$  in \eqref{fou} from now on. Other initial value of $X_0$ and other parameter value of $\sigma$ can be treated exactly in the same way.  


We emphasize again that all the previous result concerning the Berry-Ess\'een type upper bound for the parameters estimate problem of the fractional Ornstein-Uhlenbeck process is under the assumption of continuous observations or discrete observations with the step size $h_n\to 0$ as $n\to\infty$, see \cite{chenkl2020,DEKN 22,SV 18} for example. The first contribution of the present paper is to derived the upper bound of the Kolmogorov distance between $\sqrt{n}(\hat{\theta}_n-\theta)$ and its limit distribution. We state it as follows. 
\begin{theorem}\label{th1}
	Assume that $H\in(0,\frac34)$ and the fractional Ornstein-Uhlenbeck process $\{X_t:t \geq 0\}$ is defined as in \eqref{fou}. 
	If the process is observed at discrete time instants $t_k=kh,\,k=1,2,\dots,n$ and the estimator $\hat{\theta}_n$ is given by \eqref{theta hat}, then there exists a positive constant $C_{\theta,H,h}$ independent of $n$ such that when $n$ is sufficiently large, 
	\begin{equation}\label{fou B-S bound}
		\begin{split}
			d_{Kol}\left(\sqrt{n}(\hat{\theta}_{n}-\theta), \mathcal{N}\right) \leq C_{\theta,H,h}\times \left\{
			\begin{array}{ll}     
				\frac{1}{\sqrt{n}},& \quad \text{if }  H\in (0,\,\frac58],\\
				\frac{1}{n^{3-4H}},& \quad \text{if }  H\in (\frac58,\,\frac34),
			\end{array}
			\right.
		\end{split}
	\end{equation}
	where the normal random variable  $\mathcal{N} \sim N(0,\sigma_1^2)$ with $\sigma_1^2=\frac{\theta^2 \sigma_B^2}{4H^2a^2}$ and $a=H\Gamma(2H)\theta^{-2H}$, $\sigma_B^2=2\sum_{j=-\infty}^{+\infty}\rho_0^2(jh)<+\infty$. 
\end{theorem}

\begin{remark}\label{rem1}
	When $H=\frac34$, the scaling before $(\hat{\theta}_{n}-\theta)$ is $\sqrt{\frac{n}{\log n}}$ and the corresponding upper bound is $\frac1{\log n}$. However, if $H>\frac34$, the central limit theorem about the convergence of parameter will be no longer satisfy. We refer the reader to Hu et al. \cite{Hu2019}, Haress and Hu \cite{HH21} and Chen et al. \cite{chenkl2020} for details.
\end{remark}
The assumption of the step size $h_n\to 0$ as $n\to\infty$ in previous literature \cite{DEKN 22} is due to their method by which the result of the discrete observation is  transitioned from that of the continuous observation. The idea of \cite{HH21} is to deal with the double Wiener chaos random variable $W_n$ (see below) concerning the discrete observation directly. Our proof follows this idea.

The second aim of the present paper is to extend Theorem~\ref{th1} to the Ornstein-Uhlenbeck models driven by some well-known Gaussian noise such as the sub-fractional Brownian motion, the bi-fractional Brownian motion, and the sub-bifractional Brownian motion.  Now, let us recall these fractional Gaussian processes firstly.
\begin{example}\label{subfBM}
	The sub-fractional Brownian motion $\{S^H(t), t \geq 0\}$ with parameter $H\in (0,1)$ has the covariance function
	$$R(t,s)=s^{2H}+t^{2H}-\frac{1}{2}\left((s+t)^{2H}+|t-s|^{2H}\right).$$
\end{example}

\begin{example}\label{bi fBM}
	The bi-fractional Brownian motion $\{B^{H',K}(t), t\geq 0\}$ with parameters {$H'\in (0, 1),K \in (0, 2)$ and $H:=H'K\in (0,1)$} has the covariance function
	$$R(t,s)=\frac{1}{2^K}\left((s^{2H'}+t^{2H'})^K - |t-s|^{2H'K}\right).$$
\end{example}

\begin{example}\label{subbi fBM}
	The covariance function of the generalized sub-fractional Brownian motion (also known as the sub-bifractional Brownian motion), $S^{H', K}(t)$, with parameters $H' \in (0, 1)$ and $K \in(0,2)$, such that $H:=H'K\in (0,1)$, is given by:
	$$ R(s,\, t)= (s^{2H'}+t^{2H'})^{K}-\frac12 \left[(t+s)^{2H'K} + \abs{t-s}^{2H'K} \right].$$
	When $K=1$, it degenerates to the sub-fractional Brownian motion $S^H(t)$. Some properties of the process for $K\in (0,1)$ and $K\in (1,2)$ have been studied in \cite{NoutyJourn2013,Sgh 2014}.
\end{example}
\begin{example}\label{exmp7-1zili}
	The generalized fractional Brownian motion is an extension of both fBm and sub-fractional Brownian motion. Its covariance function is given by:
	$$R(s,\, t)=\frac{(a+b)^2}{2(a^2+b^2)}(s^{2H}+t^{2H})-\frac{ab}{a^2+b^2}(s+t)^{2H}-\frac12 \abs{t-s}^{2H},$$
	where $H\in (0, 1)$ and $(a,b)\neq (0,0)$ (see \cite{Zili17}). 
\end{example}

The strategy we will use is not to derive the upper Berry-Ess\'een bound for the Ornstein-Uhlenbeck models driven by the above four types of Gaussian noises one by one. We will use a more general condition in terms of the covariance functions of the Gaussian noise cited from \cite{CLZ24,CZ21} and show that the desired upper Berry-Ess\'een bound holds for all the Ornstein-Uhlenbeck models driven by the type of general Gaussian noise. 

Let us rewrite the Ornstein-Uhlenbeck model as $(Z_t)_{t\in[0,T]}$, which is the solution of the Langevin equation
\begin{equation}\label{ou2}
	\dif Z_t=-\theta Z_t \dif t+ \dif G_t, \quad t\in[0,T], \quad Z_0=0
\end{equation} 
where the driving Gaussian noise $G_t$ satisfies the following hypothesis.

\begin{hyp}\label{hyp1}
	For $H\in(0,1)$ and $H\neq \frac12$, the covariance function $R(s,t)=\mathbb{E}[G_tG_s]$ of the centered Gaussian process $(G_t)_{t\in[0,T]}$ with $G_0=0$ satisfies the following three hypotheses:
	\begin{itemize}
		\item[($H_1$)] For any fixed $s\in[0,T]$, $R(s,t)$ is an absolutely function with respect to $t$ on interval $[0,T]$.  
		\item[($H_2$)] For any fixed $t\in[0,T]$, the difference 
		\begin{equation}\label{1-1a}
			\begin{split}
				\frac{\partial R(s,t)}{\partial t}-\frac{\partial R^B(s,t)}{\partial t}	
			\end{split}
		\end{equation}
		is an absolutely continuous function with respect to $s\in[0,T]$, where $R^B(s,t)$ is the covariance function of fBm $(B^H_t)_{t\in[0,T]}$.
		\item[($H_3$)] There exists a positive constant $C$ independent of $T$ such that 
		\begin{equation}\label{1-1a1}
			\begin{split}
				\bigg|\frac{\partial}{\partial s}\bigg(\frac{\partial R(s,t)}{\partial t}-\frac{\partial R^B(s,t)}{\partial t}\bigg)\bigg|\leq C(ts)^{H-1}, 	
			\end{split}
		\end{equation}
		holds.
	\end{itemize}
\end{hyp}
It is clear that the four Gaussian noises from Example~\ref{subfBM} to Example~\ref{exmp7-1zili} satisfy Hypothesis~\ref{hyp1}, see \cite{CLZ24,CZ21}. Now we give our second contribution of the present paper. 

\begin{theorem}\label{th2}
	Assume that the Ornstein-Uhlenbeck model $\{Z_t:t \geq 0\}$ is defined as in \eqref{ou2} and that the process is observed at discrete time instants $t_k=kh,\,k=1,2,\dots,n$ and the estimator $\hat{\theta}_n$ is given by 
	\begin{equation}\label{theta hat2}
		\hat{\theta}_{n}=\left(\frac{1}{{H} \Gamma(2H ) n} \sum_{j=1}^n  Z_{jh}^2  \right)^{-\frac{1}{2H}}.
	\end{equation}
	If the driving noise satisfies Hypothesis~\ref{hyp1} with $H\in(0,\frac12)$,  then there exists a positive constant $C_{\theta,H,h}$ independent of $n$ such that when $n$ is sufficiently large, 
	\begin{equation}\label{fou B-S bound-2}
		d_{Kol}\left(\sqrt{n}(\hat{\theta}_{n}-\theta), \mathcal{N}\right) \leq C_{\theta,H,h}\times    
		\frac{1}{\sqrt{n}}
	\end{equation} where $\mathcal{N}$ is the same as in Theorem~\ref{th1}.   
\end{theorem}

The inequality \eqref{1-1a1} is good enough in some applications, see \cite{CLZ24,CZ21}. However, in some situations, a more steep inequality \eqref{phi2} is needed, see \cite{CGL24}. We write it as a new Hypothesis.   
\begin{hyp}\label{hyp1-1}
	For $H\in(0,1)$ and $H\neq \frac12$, the covariance function $R(s,t)=\mathbb{E}[G_tG_s]$ of the centered Gaussian process $(G_t)_{t\in[0,T]}$ with $G_0=0$ satisfies the above ($H_1$),   
	($H_2$)  and the following:
	\begin{itemize}
		\item[($H_3'$)] There exists a positive constants $C_1,\,C_2$ which depend only on $H',\,K$ such that the inequality
		\begin{equation}\label{phi2}
			\left| \frac{\partial}{\partial s}\left(\frac {\partial R(s,t)}{\partial t} - \frac {\partial R^{B}(s,t)}{\partial t}\right)\right| \le  C_1  (t+s)^{2H-2}+C_2 (s^{2H'}+t^{2H'})^{K-2}(st)^{2H'-1}
		\end{equation}
		holds, where $H'\in (\frac12,1),\, K\in (0,2)$ and $H:=H'K\in (0, 1)$.
	\end{itemize}
\end{hyp}	
Clearly, both Example~\ref{subfBM} and Example~\ref{exmp7-1zili} satisfy Hypothesis~\ref{hyp1-1}.  An additional requirement $H'\in (\frac12,1)$ for both Example~\ref{subbi fBM} and Example~\ref{bi fBM} makes Hypothesis~\ref{hyp1-1} hold.

\begin{theorem}\label{th2-2}
	Assume that both the  Ornstein-Uhlenbeck model $\{Z_t:t \geq 0\}$ and the estimator $\hat{\theta}_n$ are given as in Theorem~\ref{th2}. 
	If the driving noise satisfies Hypothesis~\ref{hyp1-1} with $H\in(\frac12,\frac34)$, then there exists a positive constant $C_{\theta,H,h}$ independent of $n$ such that when $n$ is sufficiently large, 
	\begin{equation}\label{fou B-S bound-3}
		d_{Kol}\left(\sqrt{n}(\hat{\theta}_{n}-\theta), \mathcal{N}\right) \leq C_{\theta,H,h}\times    
		\left\{
		\begin{array}{ll}     
			\frac{1}{\sqrt{n}},& \quad \text{if }  H\in (\frac12,\,\frac58],\\
			\frac{1}{n^{3-4H}},& \quad \text{if }  H\in (\frac58,\,\frac34),
		\end{array}
		\right.
	\end{equation} where $\mathcal{N}$ is the same as in Theorem~\ref{th1}.   
\end{theorem}
\begin{remark}
	The assumption of $H'\in (\frac12,1)$ in Hypothesis ($H_3'$) rules out $H'\in (0,\frac12]$, which is not an essential but only a technical  requirement.
\end{remark}

Based on Theorem~\ref{th2} and Theorem~\ref{th2-2}, we can finish the second aim of the present paper as follows.
\begin{corollary}
	Assume that the  Ornstein-Uhlenbeck model $\{Z_t:t \geq 0\}$ is driven by the sub-fractional Brownian motion, the bi-fractional Brownian motion, the sub-bifractional Brownian motion or the generalized fractional Brownian motion, and that the estimator $\hat{\theta}_n$ is given as in Theorem~\ref{th2}. If 
	an additional requirement $H'\in (\frac12,1)$ holds for both the bi-fractional Brownian motion and the sub-bifractional Brownian motion, then   there exists a positive constant $C_{\theta,H,h}$ independent of $n$ such that when $n$ is sufficiently large, 
	\begin{equation}\label{fou B-S bound-3-coro}
		d_{Kol}\left(\sqrt{n}(\hat{\theta}_{n}-\theta), \mathcal{N}\right) \leq C_{\theta,H,h}\times    
		\left\{
		\begin{array}{ll}     
			\frac{1}{\sqrt{n}},& \quad \text{if }  H\in (0,\,\frac58],\\
			\frac{1}{n^{3-4H}},& \quad \text{if }  H\in (\frac58,\,\frac34),
		\end{array}
		\right.
	\end{equation} where $\mathcal{N}$ is the same as in Theorem~\ref{th1}. 
\end{corollary}

The paper is organized as follows. In Section~\ref{sec2}, we recall some known results of stochastic  analysis. Proof of Theorem~\ref{th1} is given in  Section~\ref{sec.1.3}. Proof of Theorem~\ref{th2} and Theorem~\ref{th2-2} are given in  Section~\ref{sec3}. To make the paper more readable, we delay some technical calculations in Appendix.

\section{Preliminaries} \label{sec2}
\setcounter{section}{2} \setcounter{equation}{0}

This section provides a concise overview of foundational  elements about Gaussian stochastic  analysis and the Berry-Ess\'een type upper bound quantifying the distance of two normal random variables. Given a complete probability space $(\Omega, \mathcal{F}, P)$, we denote by  $\{G_t: t\in[0,T ]\}$  a continuous centered Gaussian process on this space with covariance function
$$\mathbb{E}(G_tG_s)=R(t,s), \quad s,t\in[0,T]. $$
Let $\mathfrak{H}$ be the associated reproducing kernel Hilbert space of the Gaussian process $G$, which is defined as the closure of the space of all real-valued step functions on $[0, T]$, equipped with the inner product
\begin{align*}
	\langle \mathbbm{1}_{[a,b]},\,\mathbbm{1}_{[c,d]}\rangle_{\FH}=\E\left(( G_b-G_a) ( G_d-G_c) \right),
\end{align*}
for any $0 \leq a < b \leq T$ and $0 \leq c < d \leq T$.
Denote $\{G(h) :h\in\mathfrak{H}\}$ by the  isonormal Gaussian process on the above probability space $(\Omega, \mathcal{F}, P)$ with following representation
\begin{align}\label{isoG}
	G(h)=\int_{[0,T]}h(t)\dif G_t, \ \forall \ h \in \mathfrak{H},
\end{align}
which is indexed by the elements in the Hilbert space $\mathfrak{H}$ and satisfies It\^{o}'s isometry:
\begin{align}\label{G isometry}
	\mathbb{E}\left[G(g)G(h)\right] = \langle g, h \rangle_{\mathfrak{H}}, \quad
	\forall \ g, h \in \mathfrak{H}.
\end{align}

The key point lies in establishing  the explicit formulas for the inner product in the Hilbert space $\mathfrak{H}$, which  follow the idea of \cite{CLSG,CZ21}. To elaborate  it, we first define the covariance function of  fBm $B^H$ by $R^B(s,t)=\E[B^H_sB^H_t]$, and subsequently denote the associated canonical Hilbert space by $\mathfrak{H}_1$ throughout the paper. When $H\in(\frac12,1)$ or the Lebesgue measure of intersection of the supports about two function $f, g\in\mathfrak{H}$ is  zero,  Mishura \cite{Mishura} provides that
$$\langle g,h\rangle_{\mathfrak{H}_1}
=H(2H-1)\int_{\mathbb{R}^2}g(u)h(v)|u-v|^{2H-2}dudv.$$
Suppose that $\mathcal{V}_{[0,T]}$ is the set of functions of bounded variation in $[0,T]$, and by $\mathcal{B}([0,T ])$ the Borel $\sigma$-algebra on $[0,T]$. When $H \in (0, \frac{1}{2})$, for any two functions in the set $\mathcal{V}_{[0,T]}$,  Chen et al. \cite{Alazemi2024,CLZ24} propose  a new inner product in the Hilbert space $\mathfrak{H}_1$ as following,
\begin{align}
	\langle f,\,g \rangle_{\FH_1}
	&=H \int_{[0,T]^2}  f(t)  \abs{t-s}^{2H-1}\sgn(t-s) \dif t \nu_{g}(\dif s),\quad  \forall f,\, g\in \mathcal{V}_{[0,T]}, \label{innp fg3-0}
\end{align}   
where $\nu_{g}(\dif s):=\dif \nu_{g}(s)$, and $\nu_g$ is the restriction on $\left({[0,T]},\mathcal{B}({[0,T]})\right)$ of the signed \textnormal{Lebesgue-Stieljes} measure $\mu_{g^0}$ on $\left(\Rnum,\mathcal{B}(\Rnum)\right)$, where $g^0(x)$ is defined by 
\begin{equation*}
	g^0(x)=\left\{
	\begin{array}{ll}
		g(x), & \quad \text{if}~x\in [0,T],\\
		0, &\quad \text{otherwise}.
	\end{array}
	\right.
\end{equation*}
Furthermore, if $  g'(\cdot) $ is interpreted as the distributional derivative of $g(\cdot)$, the formula \eqref{innp fg3-0} admits the following representation:
\begin{align} 
	\langle f,\,g \rangle_{\FH_1}
	&=H \int_{[0,T]^2}  f(t) g'(s) \abs{t-s}^{2H-1}\sgn(t-s) \dif t  \dif s,\quad  \forall f,\, g\in \mathcal{V}_{[0,T]}. \label{innp fg3-00}
\end{align}

Next, for the general Gaussian process $G$ and the associated reproducing kernel Hilbert space $\mathfrak{H}$, if any two functions $ f,\, g\in \mathcal{V}_{[0,T]}$, Jolis \cite{Jolis 2007}  gives a  inner product formula in Theorem 2.3 as following,
\begin{align}\label{jolis 00}
	\langle f,\,g \rangle_{\FH}=\int_{[0,T]^2}  R(s,t) \dif \big(\nu_{f}\times \nu_{g}\big)(s, t),
\end{align}
where $\nu_{g}$ is same as in equation \eqref{innp fg3-0}.
Then, under  Hypotheses $(H_1)$-$(H_2)$,  the relationship between the inner products of two functions in the Hilbert spaces $\mathfrak{H}$ and $\mathfrak{H}_1$ satisfies that
\begin{align} \label{innp fg3-zhicheng0-0}
	\langle f,\,g \rangle_{\FH}-\langle f,\,g \rangle_{\FH_1}=\int_0^T f(t)\dif t \int_0^T g(s) \frac{\partial}{\partial s}\left( \frac{\partial R(s,t)}{\partial t} -  \frac{\partial R^B(s,t)}{\partial t} \right) \dif s.
\end{align}
Moreover, when the intersection of these two functions' supports is of Lebesgue measure zero, we have
\begin{align} \label{innp fg3-zhicheng0}
	\langle f,\,g \rangle_{\FH}=\int_{[0,T]^2}  f(t)g(s) \frac{\partial^2 R(t,s)}{\partial t\partial s} \dif t   \dif s. 
\end{align} 

Finally, we introduce the Berry–Ess\'een bounds (so-called ``Stein's method'') estimating the distance between two probability distributions.  Recall that the Kolmogorov distance between two random variables $\xi,\eta$ as
$$d_{Kol}(\xi, \eta):=\sup_{z\in \Rnum}\abs{P\left(\xi\le z\right)-P\left(\eta\le z\right)}.$$	
Let the function \begin{equation}\label{yfx}
	y=f(x)=\left(\frac{1}{{H} \Gamma(2H ) } x \right)^{-\frac{1}{2H}}.	
\end{equation} Its inversion function is 
\begin{equation}\label{finversion}
	x:=g(y)=f^{-1}(y)=H\Gamma(2H)y^{-2H}.  
\end{equation}       
If $X\geq 0$ almost surely, the following lemma provides an estimate of the Kolmogorov distance between
the random variable $ f(X)$ and one normal random variable by means of that between the random variable $X$ and another normal random variable (see \cite{CLZ24,CZ21,SV 18}.)
\begin{lemma}\label{lem2}
	Let $T$ be any positive real number and $\xi\sim N(0,\sigma^2_1)$ and $\eta\sim N(0, \sigma^2_2)$ and the two functions $f$ and $g$ given by \eqref{yfx} and \eqref{finversion}, respectively. 
	If a random variable $X\ge 0$ almost surely, then
	there exists a positive constant $C$ independent of $T$ such that 
	\begin{equation}\label{mubiao}
			d_{Kol}(\sqrt{T}(f(X)-\theta) ,\, \xi) \le  C\times \left(d_{Kol}(\sqrt{T}(X-\E[X]), \eta) +\sqrt{T}\abs{  \E[X] - g(\theta)}+\frac{1}{\sqrt{T}}\right),
		\end{equation}
	where the two variance $\sigma^2_2,\,\sigma^2_1$ satisfy the following relation: \begin{equation}\label{fcha relat}
		\sigma^2_2= \big(g'(\theta)\big)^2\times{\sigma^2_1} . 
	\end{equation} 
\end{lemma}
The relation \eqref{fcha relat} comes from the delta method, please refer to chapter 3 of \cite{vand}. We point that in the previous literature \cite{CLZ24,CZ21}, the random variable $X$ is taken as 
\begin{equation*}
	\frac{1}{ T} \int_0^T  Z_{t}^2 \dif t ,
\end{equation*}however, in the present paper,  we take $T=n$ and take the random variable $X$ as\begin{equation}\label{X hanyi}
	\frac1n\sum_{j=1}^n X_{jh}^2;  \text{ and }\,\,  \frac{1}{  n} \sum_{j=1}^n  Z_{jh}^2, 
\end{equation} where $\{X_t:t \geq 0\}, \{Z_t:t \geq 0\}$ are the Ornstein-Uhlenbeck model defined as in \eqref{fou} and \eqref{ou2}, respectively.


\section{Proof of Theorem~\ref{th1}}\label{sec.1.3}

In this section, the Ornstein-Uhlenbeck model is defined as in \eqref{fou}. 	By Lemma~\ref{lem2}, we need to study the property of the second moment of sample path for the fractional Ornstein-Uhlenbeck process defined as in \eqref{X hanyi}. It is convenient to introduce a new notation and rewrite it as follows:
\begin{equation}\label{B_n}
	B_n:=\frac1n\sum_{j=1}^nX_{jh}^2.
\end{equation}

Next, we will elaborate the limit of $\mathbb{E}(B_n)$ as $n$ large enough and its convergence rate.
\begin{proposition}\label{pro1}
	Let $H\in(0,1)$ and $B_n$ be defined as in \eqref{B_n}. When $n$  large enough, the exist a constant $C$ independent of $n$ such that
	\begin{equation}\label{1}
		\abs{\mathbb{E}(B_n)-a} \leq C \times \frac1{n},
	\end{equation}
	where the constant $a=g(\theta)=H\Gamma(2H)\theta^{-2H}$.
\end{proposition}

\begin{proof}
	Through  standard computations, the  fractional Ornstein-Uhlenbeck processes $X_t$, known as the solution of \eqref{fou}, admits the explicit representation:
	\begin{equation}\label{X_t}
		X_t=\int_0^t e^{-\theta(t-s)} \dif  B_s^H,\quad   t\geq 0.
	\end{equation}
	Furthermore, Let $\{Y_t,t\in\mathbb{R}\}$ represent the stationary solution  of   fractional Ornstein-Uhlenbeck processes,  expressed as 
	\begin{equation}\label{Y_t}
		Y_t=\int_{-\infty}^t e^{-\theta(t-s)} \dif  B_s^H,\quad   t\in\mathbb{R}.
	\end{equation}
	The stationary property of $Y_t$ ensures that
	\begin{equation}\label{2}
		\mathbb{E}\left(Y_t^2\right)=\mathbb{E}\left(Y_0^2\right)=a,
	\end{equation}
	where the last equality is from  Lemma 19 in Hu et al. \cite{Hu2019}. Consequently, by the definition of $B_n$ we have
	\begin{equation}\label{3}
		\abs{\mathbb{E}(B_n)-a}=\abs{\frac1n \sum_{j=1}^n\left(\mathbb{E}\left(X_{jh}^2\right)-\mathbb{E}\left(Y_{jh}^2\right)\right)} \leq \frac1{n}\sum_{j=1}^n\left(\mathbb{E}\abs{X_{jh}^2-Y_{jh}^2}\right).
	\end{equation}
	Crucially, $X_t$ and   $Y_t$ satisfy the relationship
	\begin{equation}\label{4}
		X_t=Y_t-e^{-\theta t}Y_0, \quad \forall t\geq 0.
	\end{equation}
	Then the  Cauchy-Schwarz inequality and triangle inequality yield
	\begin{equation}\label{5}
		\abs{\mathbb{E}\left(X_t^2-Y_t^2\right)}=e^{-\theta t}\abs{\mathbb{E}\left(Y_0\left(e^{-\theta t}Y_0-2Y_t\right)\right)}\leq 3ae^{-\theta t}.
	\end{equation}
	Substituting this result into \eqref{3}, we obtain the desired result.
\end{proof}

\subsection{\bf The second moment and cumulants of the random variable $W_n$}\label{sec.sub.3.1}

To establish Theorem \ref{th1}, it is necessary to derive the Berry-Ess\'{e}en bound for the random variable $W_n$ based on the idea of \cite{CLZ24,CZ21,Khalifa,SV 18}, which is a second Wiener chaos with respect to the fBm $B_t^H$ with the form 
\begin{equation}\label{W_n}
	W_n:=\sqrt{n}\left(B_n-\mathbb{E}(B_n)\right)=\frac1{\sqrt{n}}\sum_{j=1}^n\left(X_{jh}^2-\mathbb{E}\left(X_{jh}^2\right)\right).
\end{equation}
Guided by the optimal fourth moment theorem, our  analysis focuses on: 1. Estimating the limit and convergence rate of the second moment of $W_n$; 2. Establishing upper bounds for its third and fourth cumulants. These objectives are expounded in the  following two propositions. 

\begin{proposition}\label{pro2}
	Let $H\in(0,\frac34)$ and $W_n$ be defined as in \eqref{W_n}. When $n$ is large enough, the exist a constant $C$ independent of $n$ such that
	\begin{equation}\label{7}
		\begin{split}
			\abs{\mathbb{E}(W_n^2)- \sigma_B^2} \leq C \times \left\{
			\begin{array}{ll}     
				\frac{1}{n},& \quad \text{if }  H\in (0,\,\frac12],\\
				\frac{1}{n^{3-4H}},& \quad \text{if }  H\in (\frac12,\,\frac34),
			\end{array}
			\right.
		\end{split}
	\end{equation}
	where  $\sigma_B^2$ is a series given by
	\begin{equation}\label{sigma_B^2}
		\begin{split}
			\sigma_B^2=2\sum_{j=-\infty}^{+\infty}\rho_0^2(jh)<+\infty.
		\end{split}
	\end{equation}
\end{proposition}

\begin{proof}
	Firstly, we derive the convergence of above series $\sigma_B^2$ and denote 
	\begin{equation}\label{8-1}
		\begin{split}
			\rho(t,s)=\mathbb{E}(X_tX_s), \quad \rho_0(t,s)=\mathbb{E}(Y_tY_s)
		\end{split}
	\end{equation}
	by the covariance function of fractional Ornstein-Uhlenbeck processes $X_t$ and that of stationary process $Y_t$. Moreover, due to the stationary property of $Y_t$, we can write it's  covariance function as 
	\begin{equation}\label{8-2}
		\begin{split}
			\rho_0(|t-s|)=\rho_0(t,s)=\mathbb{E}(Y_tY_s), \quad \forall s,t\in\mathbb{R}.
		\end{split}
	\end{equation}
	Specially, $\rho_0(t)=\mathbb{E}(Y_tY_0)$.  From  Theorem 2.3 of  Cheridito et al.  \cite{Cheridito}, we know that when $t$ is large enough,
	\begin{equation}\label{8-3}
		\begin{split}
			\abs{\rho_0(t)} = O(|t|^{2H-2}),
		\end{split}
	\end{equation}
	which implies that the series $\sigma_B^2$ converges, i.e., \begin{equation}\label{8-4}
		\begin{split}
			\sigma_B^2=2\sum_{j=-\infty}^{+\infty}\rho_0^2(jh)<+\infty
		\end{split}  
	\end{equation}if and only if $0<H<\frac34$, please also refer to
	Lemma 6.3 of Nourdin \cite{Nourdin2012}.
	
	Secondly, according to the product formula of Wiener-It\^{o} multiple integrals,  the second moment of second Wiener chaos $W_n$ can be rewritten as
	\begin{equation}\label{8}
		\begin{split}
			\mathbb{E}(W_n^2)=\frac2n\sum_{j,l=1}^n\rho^2(jh,lh).
		\end{split}
	\end{equation}
	Then, the triangle inequlity implies that
	\begin{equation}\label{9}
		\begin{split}
			\abs{\mathbb{E}(W_n^2)- \sigma_B^2} &\leq \frac2n \sum_{j,l=1}^n \abs{\rho^2(jh,lh)-\rho_0^2(jh,lh)} + \abs{\frac2n \sum_{j,l=1}^n \rho_0^2(\abs{j-l}h)-\sigma_B^2}\\
			&:=S_1+S_2.
		\end{split}
	\end{equation}
	For the term $S_1$, the fact $\mathbb{E}\left(Y_t^2\right)=\mathbb{E}\left(Y_0^2\right)=a$ and $\sup_{t\geq0}\mathbb{E}\left(X_t^2\right)<\infty$ (see Theorem 3.1 of Balde et al. \cite{BBES 23} ) and Cauchy-Schwarz inequality imply that
	\begin{equation}\label{10}
		\begin{split}
			\abs{\rho^2(jh,lh)-\rho_0^2(jh,lh)} &= \abs{\left(\rho(jh,lh)+\rho_0(jh,lh)\right) \cdot \left(\rho(jh,lh)-\rho_0(jh,lh)\right)} \\
			&\leq C\abs{\rho(jh,lh)-\rho_0(jh,lh)}
		\end{split}
	\end{equation}
	Combining  the relationship  \eqref{4} and a well-known fact (see Theorem 2.3 of  Cheridito et al.  \cite{Cheridito}) as following
	\begin{equation}\label{11-1}
		\begin{split}
			\abs{\rho_0(t-s)} \leq  C(1+|t-s|)^{2H-2},
		\end{split}
	\end{equation}
	we have
	\begin{equation}\label{11}
		\begin{split}
			\abs{\rho(t,s)-\rho_0(t,s)} &=\abs{\mathbb{E}\left[\left(Y_t-e^{-\theta t}Y_0\right)\left(Y_s-e^{-\theta s}Y_0\right)\right]-\mathbb{E}\left(Y_tY_s\right)}\\
			&=\abs{e^{-\theta (t+s)}\mathbb{E}\left(Y_0^2\right)-e^{-\theta t}\mathbb{E}\left(Y_sY_0\right)-e^{-\theta s}\mathbb{E}\left(Y_tY_0\right)}\\
			&\leq C\left[e^{-\theta (t+s)}+e^{-\theta t}\left(1+\abs{s}\right)^{2H-2}+e^{-\theta s}\left(1+\abs{t}\right)^{2H-2}\right].
		\end{split}
	\end{equation}
	Substituting this estimation into \eqref{10} yields
	\begin{equation}\label{12}
		\begin{split}
			\abs{\rho^2(jh,lh)-\rho_0^2(jh,lh)} \leq C\left[e^{-\theta h (l+j)}+e^{-\theta jh}\left(1+l\right)^{2H-2}+e^{-\theta lh}\left(1+j\right)^{2H-2}\right].
		\end{split}
	\end{equation}
	Consequently, we obtain
	\begin{equation}\label{13}
		\begin{split}
			S_1 &\leq \frac{C}n\left[\sum_{j,l=1}^ne^{-\theta h (l+j)}+\sum_{j,l=1}^ne^{-\theta jh}\left(1+l\right)^{2H-2}\right]\\
			&\leq \frac{C}n\left[\int_1^{\infty}\int_1^{\infty}e^{-\theta h (x+y)} \dif x \dif y + \int_1^{\infty}e^{-\theta h x} \dif x \int_1^{n} y^{2H-2} \dif y \right]\\
			&\leq \frac{C}n \left[1+n^{(2H-1)\vee 0} \right] \leq C\times \left\{
			\begin{array}{ll}     
				\frac{1}{n},& \quad \text{if }  H\in (0,\,\frac12],\\
				\frac{1}{n^{2(1-H)}},& \quad \text{if }  H\in (\frac12,\,1).
			\end{array}
			\right.
		\end{split}
	\end{equation}
	For the term $S_2$, we firstly have known that if and only if $0<H<\frac34$, 
	\begin{equation}\label{14}
		\begin{split}
			\sigma_B^2=2\sum_{j=-\infty}^{+\infty}\rho_0^2(jh)<\infty.
		\end{split}
	\end{equation}
	And then, making the change of variable $k=j-l$ yields
	\begin{equation}\label{15}
		\begin{split}
			\frac1n \sum_{j,l=1}^n \rho_0^2(\abs{j-l}h) = \sum_{k=1-n}^{n-1} \rho_0^2(\abs{k}h)(1-\frac{|k|}{n}) = \sum_{k=1-n}^{n-1} \rho_0^2(\abs{k}h) - \frac1n \sum_{k=1-n}^{n-1} \rho_0^2(\abs{k}h)|k|.
		\end{split}
	\end{equation}
	Therefore, we can scale $S_2$ as following
	\begin{equation}\label{16}
		\begin{split}
			S_2= \abs{\frac2n \sum_{j,l=1}^n \rho_0^2(\abs{j-l}h)-\sigma_B^2} \leq 2\sum_{|k|\geq n}^{\infty} \rho_0^2(\abs{k}h) + \frac4n \sum_{k=1}^{n} k \rho_0^2(kh).
		\end{split}
	\end{equation}
	Since the inequality \eqref{11-1} implies $\abs{\rho_0(kh)} \leq  C(1+k)^{2H-2}$, then for $0<H<\frac34$ we have 
	\begin{equation}\label{17}
		\begin{split}
			\sum_{|k|\geq n}^{\infty} \rho_0^2(\abs{k}h) \leq  C \sum_{k=n+1}^{\infty} (1+|k|)^{2(2H-2)} \leq C \int_{n}^{\infty} x^{2(2H-2)} \dif x \leq Cn^{4H-3},
		\end{split}
	\end{equation}
	\begin{equation}\label{18}
		\begin{split}
			\frac4n \sum_{k=1}^{n} k \rho_0^2(kh) &\leq \frac{C}{n} \sum_{k=1}^{n} (1+k)^{2(2H-2)+1} \leq \frac{C}{n} \int_{n}^{\infty} x^{2(2H-2)+1} \dif x \\
			&\leq Cn^{(4H-2) \vee 0 -1} = C \times \left\{
			\begin{array}{ll}     
				\frac{1}{n},& \quad \text{if }  H\in (0,\,\frac12],\\
				\frac{1}{n^{1-2H}},& \quad \text{if }  H\in (\frac12,\,\frac34).
			\end{array}
			\right. 
		\end{split}
	\end{equation}
	As a result, we get the estimation of $S_2$ as
	\begin{equation}\label{19}
		\begin{split}
			S_2 \leq C \times \left\{
			\begin{array}{ll}     
				\frac{1}{n},& \quad \text{if }  H\in (0,\,\frac12],\\
				\frac{1}{n^{1-2H}},& \quad \text{if }  H\in (\frac12,\,\frac34).
			\end{array}
			\right. 
		\end{split}
	\end{equation}
	Substituting the estimations \eqref{13} and \eqref{19} into \eqref{9} with the fact that $2(H-1)<4H-3$ if $H\in (\frac12,\,\frac34)$, we obtain the desired result. 
\end{proof}

Next, we derive the the upper bounds of the third and fourth cumulants of $W_n$.
\begin{proposition}\label{pro3}
	Let $H\in(0,\frac34)$ and $W_n$ be defined as in \eqref{W_n}. Then for large enough $n$, we have
	\begin{equation}\label{20}
		\begin{split}
			\max \left\{ \abs{k_3(W_n)}, k_4(W_n) \right\}
			\leq C \times \left\{
			\begin{array}{ll}     
				\frac{1}{\sqrt{n}},& \quad \text{if }  H\in (0,\,\frac23],\\
				n^{\frac32(4H-3)},& \quad \text{if }  H\in (\frac23,\,\frac34).
			\end{array}
			\right.
		\end{split}
	\end{equation}
\end{proposition}
\begin{proof}
	The core approach involves  comparing the third and fourth cumulants of  $W_n$ with those of $\overline{W}_n$, which denote by a random variable
	\begin{equation}\label{W_n2}
		\overline{W}_n=\frac1{\sqrt{n}}\sum_{j=1}^n\left(Y_{jh}^2-\mathbb{E}\left(Y_{jh}^2\right)\right),
	\end{equation}
	It is clear that $\overline{W}_n$ also  belongs to the  second Wiener chaos with respect to fBm. 	Then, we apply the product formula of Wiener-It\^{o} multiple integrals to compute  the cumulants of second Wiener chaos $W_n$ as following
	\begin{equation}\label{k3W_n}
		\begin{split}
			k_3(W_n):=\mathbb{E}(W_n^3)=\frac{8}{n^{3/2}}\sum_{j,k,l=1}^n\rho(jh,kh)\rho(kh,lh)\rho(lh,jh),
		\end{split}
	\end{equation}
	\begin{equation}\label{k4W_n}
		\begin{split}
			0 < k_4(W_n)&:=\mathbb{E}(W_n^4)-3(\mathbb{E}(W_n^2))^2\\
			&=\frac{48}{n^{2}}\sum_{i,j,k,l=1}^n\rho(ih,jh)\rho(jh,kh)\rho(kh,lh)\rho(lh,jh).
		\end{split}
	\end{equation}
	The third and fourth cumulants of  $\overline{W}_n$ will be similar with  $\rho$   replaced by $\rho_0$. According to  Propositions 6.3 and 6.4 of    Bierm\'e et al. \cite{BBNP 12} and the inequality \eqref{ap1} in Lemma \ref{lem1}, we obtain that
	\begin{equation}\label{22}
		\begin{split}
			k_3(\overline{W}_n) &\leq \frac{C}{\sqrt{n}} \left(\sum_{|k|<n}\abs{\rho_0(k)}^{\frac32}\right)^2 \leq \frac{C}{\sqrt{n}} \left(\sum_{k=0}^{n-1}(1+k)^{\frac32 (2H-2)}\right)^2\\
			&\leq Cn^{(6H-4)\vee 0 - \frac12} = C \times \left\{
			\begin{array}{ll}     
				\frac{1}{\sqrt{n}},& \quad \text{if }  H\in (0,\,\frac23],\\
				n^{\frac32(4H-3)},& \quad \text{if }  H\in (\frac23,\,\frac34),
			\end{array}
			\right.
		\end{split}
	\end{equation}
	\begin{equation}\label{23}
		\begin{split}
			k_4(\overline{W}_n) &\leq \frac{C}{n} \left(\sum_{|k|<n}\abs{\rho_0(k)}^{\frac43}\right)^3 \leq \frac{C}{n} \left(\sum_{k=0}^{n-1}(1+k)^{\frac43 (2H-2)}\right)^3\\
			&\leq Cn^{(8H-5)\vee 0 - 1} = C \times \left\{
			\begin{array}{ll}     
				\frac{1}{n},& \quad \text{if }  H\in (0,\,\frac58],\\
				n^{2(4H-3)},& \quad \text{if }  H\in (\frac58,\,\frac34).
			\end{array}
			\right.
		\end{split}
	\end{equation}
	On the other hand, from the identity \eqref{4}, we rewrite $W_n$ as
	\begin{equation}\label{24}
		\begin{split}
			W_n=\overline{W}_n + \frac1{\sqrt{n}} \sum_{j=1}^n e^{-\theta jh}R_{jh} - \sqrt{n}\left(\mathbb{E}(B_n)-a \right),
		\end{split}
	\end{equation}
	where $R_{jh}=-2Y_{jh}Y_{0} + e^{-\theta jh}Y_{0}^2$, which satisfies
	\begin{equation}\label{25}
		\begin{split}
			\sup_{j}\abs{\abs{R_{j}}}_{L^2(\Omega)}<\infty,
		\end{split}
	\end{equation} 
	based on the stationary property of $Y_t$ and Cauchy-Schwarz inequality.  Combining this result with  the fact that $R_j$ is a 2-th Wiener chaos, we have 
	\begin{equation}\label{26}
		\begin{split}
			\sup_{j}\abs{\abs{\frac1{\sqrt{n}} \sum_{j=1}^n e^{-\theta jh}R_{jh}}}_{L^2(\Omega)}<\frac{C}{\sqrt{n}},
		\end{split}
	\end{equation} 
	where C is independent of $n$. Then, by the identity \eqref{24} and Cauchy-Schwarz inequality, Minkowski's inequality, and hypercontractivity
	property of Wiener chaos,  we obtain that 
	\begin{equation}\label{27}
		\begin{split}
			&\abs{k_3({W}_n)-k_3(\overline{W}_n)} = \abs{\mathbb{E}\left(W_n^3-\overline{W}_n^3\right)} \\
			&= \abs{\mathbb{E}\left[\left(\frac1{\sqrt{n}} \sum_{j=1}^n e^{-\theta jh}R_{jh} - \sqrt{n}\left(\mathbb{E}(B_n)-a \right)\right) \left( W_n^2+W_n\overline{W}_n+\overline{W}_n^2\right)\right]}\\
			&\leq \left[\abs{\abs{\frac1{\sqrt{n}} \sum_{j=1}^n e^{-\theta jh}R_{jh}}}_{L^2(\Omega)} + \abs{\abs{\sqrt{n}\left(\mathbb{E}(B_n)-a \right)}}_{L^2(\Omega)} \right] \abs{\abs{W_n^2+W_n\overline{W}_n+\overline{W}_n^2}}_{L^2(\Omega)}\\
			&\leq \frac{C}{\sqrt{n}},
		\end{split}
	\end{equation}
	where in the last inequality we also have used the estimation \eqref{26} and Propositions \ref{pro1}, \ref{pro2}.  A similar method yields
	\begin{equation}\label{28}
		\begin{split}
			&\abs{k_4({W}_n)-k_4(\overline{W}_n)} =\abs{\mathbb{E}\left(W_n^4\right) - \mathbb{E}\left(\overline{W}_n^4\right)} + 3\abs{\left(\mathbb{E}W_n^2\right)^2 - \left(\mathbb{E}\overline{W}_n^2\right)^2}\\
			&\leq \abs{\abs{\frac1{\sqrt{n}} \sum_{j=1}^n e^{-\theta jh}R_{jh} - \sqrt{n}\left(\mathbb{E}(B_n)-a \right)}}_{L^2(\Omega)} \cdot \abs{\abs{W_n^3+W_n^2\overline{W}_n++W_n\overline{W}_n^2+\overline{W}_n^3}}_{L^2(\Omega)}\\
			&\quad +3\abs{\left(\mathbb{E}W_n^2 + \mathbb{E}\overline{W}_n^2\right) \left(\mathbb{E}W_n^2 - \mathbb{E}\overline{W}_n^2\right)}\\
			&\leq \frac{C}{\sqrt{n}}.
		\end{split}
	\end{equation}
	Combining the estimations \eqref{22},  \eqref{23},  \eqref{27},  \eqref{28}, we can obtain the desired result.
\end{proof}

\subsection{{\bf Berry-Ess\'een type upper bound for the moment estimator of  fOU process}}\label{sec.sub.3.2}

In this section, we concentrate on establishing the Berry-Ess\'een type upper bound for the moment estimator of the fractional Ornstein-Uhlenbeck process under discrete observations with the fixed step size.

{\bf Proof of Theorem \ref{th1}.} 
Recall the definition of $\hat{\theta}_{n}$ and $B_n$,  
we take the random variable $X$, $f(X)$ in the Lemma \ref{lem2} as 
\begin{equation}\label{29}
	\begin{split}
		X=B_n=\frac{1}{ n} \sum_{j=1}^n  X_{jh}^2 , \quad f(X)=\hat{\theta}_{n}=\left(\frac{1}{{H} \Gamma(2H ) } B_n  \right)^{-\frac{1}{2H}},
	\end{split}
\end{equation}
and $a=g(\theta)= H\Gamma(2H)\theta^{-2H}$, $\mathcal{N}=\xi \sim N(0,\sigma_1^2)$,  $\eta=\varpi \sim N(0,\sigma_2^2)$. Section 1.3.2.2 of Kubilius et al. \cite{KMR17} shows that $B_n \rightarrow{a}$ almost surely, so we have $B_n>0$ almost surely. Then, according to  Lemma \ref{lem2}, there exists a positive constant $C$ independent of $T$ such that for $T$ large enough
\begin{equation}\label{30}
	\begin{split}
		&d_{Kol}\left(\sqrt{n}(\hat{\theta}_{n}-\theta), \mathcal{N}\right)\\
		&\le  C\times \left(d_{Kol}(\sqrt{n}(B_n-\E[B_n]), \varpi) +\sqrt{n}\abs{  \E[B_n] - a}+\frac{1}{\sqrt{n}}\right),
	\end{split}
\end{equation}
where $\varpi$ is a normal random variable with zero mean and variance  $\sigma_2^2=\sigma^2_B$ defined as in equation \eqref{sigma_B^2} and then $\sigma_1^2=\frac{\theta^2 \sigma_B^2}{4H^2a^2}$ from \eqref{fcha relat}.  

Firstly, we estimate the term $d_{Kol}(\sqrt{n}(B_n-\E[B_n]), \varpi)$.  Denote a sequence of random variables $\varpi_n\sim N(0,\sigma_n^2)$ with the variance $\sigma_n^2=\mathbb{E}(W_n^2)$, where $W_n=\sqrt{n}(B_n-\E[B_n])$ defined in \eqref{W_n}.
Then, we have that
\begin{equation}\label{32}
	\begin{split}
		d_{Kol}(W_n, \varpi) \leq d_{Kol}(W_n, \varpi_n) + d_{Kol}(\varpi_n, \varpi),
	\end{split}
\end{equation}
by the triangle inequality.
The optimal fourth moment theorem of Nourdin and Peccati \cite{NP2015}  and the well-known fact that $d_{Kol}(\cdot,\cdot) \leq d_{TV}(\cdot,\cdot)$ imply that 
\begin{equation}\label{31}
	\begin{split}
		d_{Kol}(W_n, \varpi_n) &\leq d_{TV}(W_n, \varpi_n) \leq C\max\left\{k_3(W_n), k_4(W_n)\right\}\\ &\leq C \times \left\{
		\begin{array}{ll}     
			\frac{1}{\sqrt{n}},& \quad \text{if }  H\in (0,\,\frac23],\\
			n^{\frac32(4H-3)},& \quad \text{if }  H\in (\frac23,\,\frac34),
		\end{array}
		\right.
	\end{split}
\end{equation}
where the last inequality is resulted from Proposition \ref{pro3}.   Using  Proposition 3.6.1 of Nourdin and Peccati   \cite{NP2012} and Proposition \ref{pro2}  yield
\begin{equation}\label{33}
	\begin{split}
		d_{Kol}(\varpi_n, \varpi) \leq \frac2{\sigma_n^2 \vee \sigma_B^2}\abs{\sigma_n^2 - \sigma_B^2} \leq C \times \left\{
		\begin{array}{ll}     
			\frac{1}{n},& \quad \text{if }  H\in (0,\,\frac12],\\
			\frac{1}{n^{3-4H}},& \quad \text{if }  H\in (\frac12,\,\frac34).
		\end{array}
		\right.
	\end{split}
\end{equation}
Combining this result with the inequalities \eqref{32}, \eqref{31} implies that
\begin{equation}\label{34}
	\begin{split}
		(d_{Kol}(\sqrt{n}(B_n-\E[B_n]), \varpi) \leq C 
		\times \left\{
		\begin{array}{ll}     
			\frac{1}{\sqrt{n}},& \quad \text{if }  H\in (0,\,\frac58],\\
			\frac{1}{n^{3-4H}},& \quad \text{if }  H\in (\frac58,\,\frac34).
		\end{array}
		\right.
	\end{split}
\end{equation}

Secondly, it is straightforward to show for $H\in(0,1)$
\begin{equation}\label{35}
	\sqrt{n}\abs{\mathbb{E}(B_n)-a} \leq C \times \frac1{\sqrt{n}},
\end{equation}
from Proposition \eqref{pro1}. Consequently, substituting the inequalities \eqref{34}, \eqref{35} into the estimation \eqref{30} yields the desired Berry-Ess\'een upper bound \eqref{fou B-S bound} in Theorem \ref{th1}.
\hfill$\Box$\medskip

\section{Proof of Theorem~\ref{th2} and  Theorem~\ref{th2-2}}\label{sec3}
\setcounter{equation}{0}

\subsection{\bf The second moment and cumulants of the random variable $\widetilde{W}_n$}\label{sec.sub.4.1}

Prior to proving Theorem~\ref{th2}, liking the definition of $W_n$, we first denote by $\widetilde{W}_n$ a second Wiener chaos with respect to the general Gaussian process $G_t$ as:
\begin{equation}\label{W_n3}
	\widetilde{W}_n := \sqrt{n}\left(A_n-\mathbb{E}(A_n)\right) = \frac1{\sqrt{n}}\sum_{j=1}^n\left(Z_{jh}^2-\mathbb{E}\left(Z_{jh}^2\right)\right).
\end{equation}
Establishing the Berry-Ess\'een upper bound for $\widetilde{W}_n$ constitutes a critical step in the proof of Theorem \ref{th2}. The following  proposition  characterizes the asymptotic behavior of its second moment and convergence rate.

\begin{proposition}\label{pro4}
	Let $H\in(0,\frac12)$ and $\widetilde{W}_n$ be defined as in \eqref{W_n3}. When $n$  large enough, the exist a constant $C$ independent of $n$ such that
	\begin{equation}\label{38}
		\begin{split}
			\abs{\mathbb{E}(\widetilde{W}_n^2)- \sigma_B^2} \leq C \times \frac1{\sqrt{n}},
		\end{split}
	\end{equation}
	where  $\sigma_B^2$ is a series defined in \eqref{sigma_B^2}.
\end{proposition}
\begin{proof}
	Based on  Proposition \ref{pro2}, the proof reduces to verifying that
	\begin{equation}\label{39}
		\begin{split}
			\abs{\mathbb{E}(\widetilde{W}_n^2- W_n^2)} \leq C \times \frac1{\sqrt{n}}.
		\end{split}
	\end{equation}
	Denoting the covariance  function of $Z_t$ as $\tilde{\rho}(t,s)=\mathbb{E}(Z_tZ_s)$, we rewrite the left hand of above inequality as
	\begin{equation}\label{40}
		\begin{split}
			\mathbb{E}(\widetilde{W}_n^2- W_n^2) &= \frac2n \sum_{j,l=1}^n \left[\tilde{\rho}^2(jh,lh)-\rho^2(jh,lh)\right]\\
			&=\frac2n \sum_{j,l=1}^n \left[\left( \tilde{\rho} (jh,lh) - \rho(jh,lh) + \rho(jh,lh) \right)^2-\rho^2(jh,lh)\right]\\
			&=\frac2n \sum_{j,l=1}^n \left( \tilde{\rho} (jh,lh) - \rho(jh,lh)  \right)^2 + \frac4n \sum_{j,l=1}^n \rho(jh,lh) \left( \tilde{\rho} (jh,lh) - \rho(jh,lh)  \right)\\
			&:=D_1+D_2.
		\end{split}
	\end{equation}
	Next, we will estimate the upper bound of $D_1$ and $D_2$, respectively. The estimation \eqref{ap2} in Lemma \ref{lem1}, along with the symmetry of $j,l$ and the change variable $k=l-j$ imply that
	\begin{equation}\label{41}
		\begin{split}
			0\leq D_1 &\leq \frac{C}n \sum_{j,l=1}^n \left[\left( 1+(j \wedge l)  \right)^{2(H-1)} \wedge   \left( 1+\abs{j-l}  \right)^{H-1} \right]^2\\
			&\leq \frac{C}n \sum_{1 \leq j \leq l \leq n} \left[\left( 1+j  \right)^{2(H-1)} \wedge   \left( 1+(l-j)  \right)^{H-1} \right]^2\\
			&\leq \frac{C}n \sum_{1 \leq j \leq l \leq n} \left( 1+j  \right)^{2(H-1)} \cdot   \left( 1+(l-j)  \right)^{H-1} \\
			&\leq \frac{C}n \sum_{j=1}^n \left( 1+j  \right)^{2(H-1)} \sum_{k=0}^{n-1} \left( 1+k  \right)^{H-1}\\
			&\leq \frac{C}nn^{(2H-1)\vee 0}n^H=Cn^{H-1} \leq C\frac1{\sqrt{n}},
		\end{split}
	\end{equation}
	where the last two inequalities are from condition $H\in(0,\frac12)$. Using the triangle inequality,  estimations \eqref{ap1}, \eqref{ap2} in Lemma \ref{lem1} and the symmetry of $j,l$ yield that 
	\begin{equation}\label{42}
		\begin{split}
			\abs{D_2} &\leq \frac{C}n \sum_{j,l=1}^n \left( 1+\abs{j-l}  \right)^{2(H-1)} \cdot \left[\left( 1+(j \wedge l)  \right)^{2(H-1)} \wedge   \left( 1+\abs{j-l}  \right)^{H-1} \right]\\
			&\leq \frac{C}n \sum_{1 \leq j \leq l \leq n}\left( 1+(l-j)  \right)^{2(H-1)} \cdot \left( 1+j  \right)^{2(H-1)} \\
			&\leq \frac{C}n \left(\sum_{j=1}^n \left( 1+j  \right)^{2(H-1)}\right)^2 \leq \frac{C}n n^{(4H-2)\vee 0} =\frac{C}n.
		\end{split}
	\end{equation}
	Consequently, substituting inequality \eqref{41}, \eqref{42} into equation \eqref{40} obtains the desired result \eqref{39}. In summary, this completes the proof.
\end{proof}

Next, we focus on estimating  the third and fourth  cumulants of random variable $\widetilde{W}_n$.
\begin{proposition}\label{pro5}
	Let $H\in(0,\frac12)$ and $\widetilde{W}_n$ be defined as in \eqref{W_n3}. Denote the third  cumulants of random variable $\widetilde{W}_n$ by
	\begin{equation}\label{k3tildeW_n}
		\begin{split}
			k_3(\widetilde{W}_n):=\mathbb{E}(\widetilde{W}_n^3)=\frac{8}{n^{3/2}}\sum_{j,k,l=1}^n\tilde{\rho}(jh,kh)\tilde{\rho}(kh,lh)\tilde{\rho}(lh,jh).
		\end{split}
	\end{equation}
	When $n$  large enough, the exist a constant $C$ independent of $n$ such that
	\begin{equation}\label{43}
		\begin{split}
			\abs{k_3(\widetilde{W}_n)} \leq C \times \frac1{\sqrt{n}}.
		\end{split}
	\end{equation}
\end{proposition}
\begin{proof}
	According to the estimations \eqref{22} and \eqref{27} in Proposition \ref{pro3}, we only need to show
	\begin{equation}\label{44}
		\begin{split}
			\abs{k_3(\widetilde{W}_n) - k_3(W_n)} \leq C \times \frac1{\sqrt{n}},
		\end{split}
	\end{equation}
	The above inequality is equivalent to
	\begin{equation}\label{45}
		\begin{split}
			I:=\abs{\sum_{j,k,l=1}^n \Big[ \tilde{\rho}(jh,kh)\tilde{\rho}(kh,lh)\tilde{\rho}(lh,jh) - \rho(jh,kh)\rho(kh,lh)\rho(lh,jh) \Big]} \leq C n.
		\end{split}
	\end{equation}
	For the sake of  simplicity, we denote $x=\tilde{\rho}(jh,kh) - \rho(jh,kh), y=\tilde{\rho}(kh,lh) - \rho(kh,lh), z=\tilde{\rho}(lh,jh) - \rho(lh,jh)$, then $I$ is decomposed into the following seven summations,  
	\begin{equation}\label{46}
		\begin{split}
			I&=\Bigg|\sum_{j,k,l=1}^n \Big[ x\rho(kh,lh)\rho(lh,jh) + \rho(jh,kh)y\rho(lh,jh) + \rho(jh,kh)\rho(kh,lh)z\\
			& \qquad\qquad\quad+xy\rho(lh,jh) + x\rho(kh,lh)z + \rho(jh,kh)yz + xyz \Big]\Bigg|\\
			&:=\abs{\sum_{i=1}^7I_i}.
		\end{split}
	\end{equation}
	Next, we estimate the upper bound for each of $I_i,i=1,\cdots,7$. The key point is to select the scaling approach of $x,y,z$ based on the different symmetries of $i,j,k$ in every $I_i$. The  estimations \eqref{ap1}, \eqref{ap2} in Lemma \ref{lem1} and the symmetry of $j,k$ 
	imply that
	\begin{equation}\label{47}
		\begin{split}
			\abs{I_1} & \leq C \sum_{j,k,l=1}^n \abs{x\rho(kh,lh)\rho(lh,jh)}\\
			&\leq C \sum_{j,k,l=1}^n \left( 1+(j \wedge k)  \right)^{2(H-1)} \cdot   \left( 1+\abs{k-l}  \right)^{2(H-1)} \cdot   \left( 1+\abs{l-j}  \right)^{2(H-1)}\\
			&\leq C \sum_{1\leq j \leq k \leq n,1 \leq l \leq n} \left( 1+j  \right)^{2(H-1)} \cdot   \left( 1+\abs{k-l}  \right)^{2(H-1)} \cdot   \left( 1+\abs{l-j}  \right)^{2(H-1)}\\
			&\leq C,
		\end{split}
	\end{equation}
	where in the last inequality we use  Lemma \ref{lem1a} with the condition $H\in(0,\frac12)$.\\
	With the similar way, we also have
	\begin{equation}\label{48}
		\begin{split}
			\abs{I_2} \leq C \sum_{j,k,l=1}^n    \left( 1+\abs{j-k}  \right)^{2(H-1)}  \cdot \left( 1+(k \wedge l)  \right)^{2(H-1)} \cdot   \left( 1+\abs{l-j}  \right)^{2(H-1)}\leq C,
		\end{split}
	\end{equation}
	\begin{equation}\label{49}
		\begin{split}
			\abs{I_3} \leq C \sum_{j,k,l=1}^n \left( 1+\abs{j-k}  \right)^{2(H-1)} \cdot   \left( 1+\abs{k-l}  \right)^{2(H-1)} \cdot \left( 1+(l \wedge j)  \right)^{2(H-1)}    \leq C.
		\end{split}
	\end{equation}
	According to  Lemma \ref{lem1} and the symmetry of $l,j$, we can derive that
	\begin{equation}\label{50}
		\begin{split}
			\abs{I_4} & \leq C \sum_{j,k,l=1}^n \abs{xy\rho(lh,jh)}\\
			&\leq C \sum_{j,k,l=1}^n \left( 1+\abs{j-k}  \right)^{H-1} \cdot   \left( 1+\abs{k-l}  \right)^{H-1} \cdot \left( 1+\abs{l-j}  \right)^{2(H-1)}\\
			&\leq C \sum_{1\leq l \leq j \leq n,1 \leq k \leq n} \left( 1+\abs{j-k}  \right)^{H-1} \cdot   \left( 1+\abs{k-l}  \right)^{H-1} \cdot \left( 1+(j-l)  \right)^{2(H-1)}\\
			&\leq Cn^{2H}\leq Cn,
		\end{split}
	\end{equation}
	where in the last inequality we use  Lemma \ref{lem1a} with the condition $H\in(0,\frac12)$.\\
	With the similar way, we also have
	\begin{equation}\label{51}
		\begin{split}
			\abs{I_5} \leq C \sum_{j,k,l=1}^n \left( 1+\abs{j-k}  \right)^{H-1} \cdot   \left( 1+\abs{k-l}  \right)^{2(H-1)} \cdot \left( 1+\abs{l-j}  \right)^{H-1}&\leq Cn^{2H}\leq Cn,
		\end{split}
	\end{equation}
	\begin{equation}\label{52}
		\begin{split}
			\abs{I_6} \leq C \sum_{j,k,l=1}^n \left( 1+\abs{j-k}  \right)^{2(H-1)} \cdot   \left( 1+\abs{k-l}  \right)^{H-1} \cdot \left( 1+\abs{l-j}  \right)^{H-1}&\leq Cn^{2H}\leq Cn.
		\end{split}
	\end{equation}
	For the last term $I_7$, we use Lemma \ref{lem1} and the symmetry of $j,k,l$ obtain that
	\begin{equation}\label{53}
		\begin{split}
			\abs{I_7} & \leq C \sum_{j,k,l=1}^n \abs{xyz}\leq C \sum_{j,k,l=1}^n \left( 1+(j \wedge k)  \right)^{2(H-1)} \cdot \left( 1+(k \wedge l)  \right)^{2(H-1)} \cdot \left( 1+\abs{l-j}  \right)^{H-1}\\
			&\leq C \sum_{1\leq j \leq k \leq l \leq n} \left( 1+j \right)^{2(H-1)} \cdot \left( 1+k  \right)^{2(H-1)} \cdot \left( 1+(l-j)  \right)^{H-1} \leq Cn^H \leq Cn,
		\end{split}
	\end{equation}
	where  the last inequality is from Lemma \ref{lem1a} with the condition $H\in(0,\frac12)$. In conclusion, we get 
	\begin{equation}\label{54}
		\begin{split}
			0\leq I \leq Cn.
		\end{split}
	\end{equation}
	This complete the proof.
\end{proof}
\begin{proposition}\label{pro6}
	Let $H\in(0,\frac12)$ and $\widetilde{W}_n$ be defined as in \eqref{W_n3}. Denote the forth  cumulants of random variable $\widetilde{W}_n$ by
	\begin{equation}\label{k4tildeW_n}
		\begin{split}
			0 < k_4(\widetilde{W}_n)&:=\mathbb{E}(\widetilde{W}_n^4)-3(\mathbb{E}(\widetilde{W}_n^2))^2\\
			&=\frac{48}{n^{2}}\sum_{i,j,k,l=1}^n\tilde{\rho}(ih,jh)\tilde{\rho}(jh,kh)\tilde{\rho}(kh,lh)\tilde{\rho}(lh,jh).
		\end{split}
	\end{equation}
	When $n$ is large enough, the exist a constant $C$ independent of $n$ such that
	\begin{equation}\label{55}
		\begin{split}
			k_4(\widetilde{W}_n) \leq C \times \frac1{\sqrt{n}}.
		\end{split}
	\end{equation}
\end{proposition}
\begin{proof}
	From the previous estimation concerning $k_4(\overline{W}_n)$ and $\abs{k_4({W}_n)-k_4(\overline{W}_n)}$ in Proposition \ref{pro3}, it is sufficient to prove that
	\begin{equation}\label{56}
		\begin{split}
			\abs{k_4(\widetilde{W}_n) - k_4(W_n)} \leq C \times \frac1{\sqrt{n}},
		\end{split}
	\end{equation}
	which is equivalent to showing
	\begin{equation}\label{57}
		\begin{split}
			J&:=\abs{\sum_{j,k,l,m=1}^n \Big[ \tilde{\rho}(jh,kh)\tilde{\rho}(kh,lh)\tilde{\rho}(lh,mh)\tilde{\rho}(mh,jh) - \rho(jh,kh)\rho(kh,lh)\rho(lh,mh)\rho(mh,jh) \Big]}\\
			&\leq C n^{\frac32}.
		\end{split}
	\end{equation}
	Denoting by  $x,y$ the same symbol as  in Proposition \ref{pro5} and $z=\tilde{\rho}(lh,mh) - \rho(lh,mh), w=\tilde{\rho}(mh,jh) - \rho(mh,jh)$, which implies that $J$ can be decomposed into fifteen summations:
	\begin{equation}\label{57a}
		\begin{split}
			J&=\Bigg|\sum_{j,k,l,m=1}^n \Big[ x\rho(kh,lh)\rho(lh,mh)\rho(mh,jh) + \rho(jh,kh)y\rho(lh,mh)\rho(mh,jh)\\
			& \qquad\qquad\quad+ \rho(jh,kh)\rho(kh,lh)z\rho(mh,jh) + \rho(jh,kh)\rho(kh,lh)\rho(lh,mh)w\\
			& \qquad\qquad\quad+xy\rho(lh,mh)\rho(mh,jh) + x\rho(kh,lh)z\rho(mh,jh) + x\rho(kh,lh)\rho(lh,mh)w\\
			& \qquad\qquad\quad + \rho(jh,kh)yz\rho(mh,jh) + \rho(jh,kh)y\rho(lh,mh)w + \rho(jh,kh)\rho(kh,lh)zw\\
			& \qquad\qquad\quad+ xyz\rho(mh,jh) + xy\rho(lh,mh)w + x\rho(kh,lh)zw + \rho(jh,kh)yzw +xyzw\Big]\Bigg|\\
			&:=\Bigg|\sum_{i=1}^{15}J_i\Bigg|.
		\end{split}
	\end{equation}
	We divide the fifteen summations into five groups and discuss them separately. The idea is to analyze each term $J_i$ by the different symmetry of the sum index in $J_i$. 
	
	Group 1. Lemma \ref{lem1} and the symmetry of $j,k$ imply that
	\begin{equation}\label{58}
		\begin{split}
			&\abs{J_1} = \abs{\sum_{j,k,l,m=1}^n x\rho(kh,lh)\rho(lh,mh)\rho(mh,jh)} \\
			&\leq C \sum_{j,k,l,m=1}^n \left( 1+(j \wedge k)  \right)^{2(H-1)} \cdot   \left( 1+\abs{k-l}  \right)^{2(H-1)} \cdot \left( 1+\abs{l-m}  \right)^{2(H-1)} \cdot \left( 1+\abs{m-j}  \right)^{2(H-1)}\\
			&\leq C \sum_{1\leq j \leq k \leq n,1 \leq l,m \leq n} \left( 1+j  \right)^{2(H-1)} \cdot \left( 1+\abs{k-l}  \right)^{2(H-1)} \cdot \left( 1+\abs{l-m}  \right)^{2(H-1)} \cdot \left( 1+\abs{m-j}  \right)^{2(H-1)}\\
			&\leq C, 
		\end{split}
	\end{equation}
	where in the last inequality we use  Lemma \ref{lem1a} with the condition $H\in(0,\frac12)$. Similarly, using symmetry and analogous arguments, we obtain:
	\begin{equation}\label{59}
		\begin{split}
			\abs{J_i} \leq C, \quad j=2,3,4.
		\end{split}
	\end{equation}
	
	Group 2. For the term $J_5$, noticing the symmetry of $j,l$ and utilizing  Lemma \ref{lem1}, we have 
	\begin{align}
		\notag &\abs{J_5} = \abs{\sum_{j,k,l,m=1}^n xy\rho(lh,mh)\rho(mh,jh)} \\
		\notag &\leq C \sum_{j,k,l,m=1}^n \left( 1+(j \wedge k)  \right)^{2(H-1)} \cdot   \left( 1+\abs{k-l}  \right)^{H-1} \cdot \left( 1+\abs{l-m}  \right)^{2(H-1)} \cdot \left( 1+\abs{m-j}  \right)^{2(H-1)}\\
		\notag &\leq C \Bigg[ \sum_{1\leq k \leq j \leq l \leq n,1 \leq m \leq n} \left( 1+k  \right)^{2(H-1)} \cdot \left( 1+ (l-k)  \right)^{H-1} \cdot \left( 1+\abs{l-m}  \right)^{2(H-1)} \cdot \left( 1+\abs{m-j}  \right)^{2(H-1)}\\
		\notag &+ \sum_{1\leq j \leq k \leq l \leq n,1 \leq m \leq n} \left( 1+j  \right)^{2(H-1)} \cdot \left( 1+ (l-k)  \right)^{H-1} \cdot \left( 1+\abs{l-m}  \right)^{2(H-1)} \cdot \left( 1+\abs{m-j}  \right)^{2(H-1)}\\
		\notag &+ \sum_{1\leq j \leq l \leq k \leq n,1 \leq m \leq n} \left( 1+j  \right)^{2(H-1)} \cdot \left( 1+ (k-l)  \right)^{H-1} \cdot \left( 1+\abs{l-m}  \right)^{2(H-1)} \cdot \left( 1+\abs{m-j}  \right)^{2(H-1)} \Bigg]\\
		&\leq C n^H, 
	\end{align}\label{60}
	where the last inequality is caused by Lemma \ref{lem1a} with the condition $H\in(0,\frac12)$. Notice that $J_i, j=7,8,10$ share the similar symmetry with $J_5$, which implies that
	\begin{equation}\label{61}
		\begin{split}
			\abs{J_i} \leq C n^H, \quad j=7,8,10.
		\end{split}
	\end{equation}
	
	Group 3. We estimate the term $J_6$ by the symmetry of $(j,k)$ and $(l,m)$ with  Lemma \ref{lem1},
	\begin{equation}\label{62}
		\begin{split}
			&\abs{J_6} = \abs{\sum_{j,k,l,m=1}^n x\rho(kh,lh)z\rho(mh,jh)} \\
			&\leq C  \sum_{j,k,l,m=1}^n \left( 1+(j \wedge k)  \right)^{2(H-1)} \cdot   \left( 1+\abs{k-l}  \right)^{2(H-1)} \cdot \left( 1+\abs{l-m}  \right)^{H-1} \cdot \left( 1+\abs{m-j}  \right)^{2(H-1)}\\
			&\leq C \Bigg[ \sum_{1\leq j \leq k  \leq n,1 \leq l \leq m \leq n} \left( 1+j  \right)^{2(H-1)} \cdot   \left( 1+\abs{k-l}  \right)^{2(H-1)} \cdot \left( 1+(m-l)  \right)^{H-1} \cdot \left( 1+\abs{m-j}  \right)^{2(H-1)}\\
			&+ \sum_{1\leq j \leq k  \leq n,1 \leq m \leq l \leq n} \left( 1+j  \right)^{2(H-1)} \cdot   \left( 1+\abs{k-l}  \right)^{2(H-1)} \cdot \left( 1+(l-m)  \right)^{H-1} \cdot \left( 1+\abs{m-j}  \right)^{2(H-1)}\Bigg]\\
			&\leq C n^H, 
		\end{split}
	\end{equation}
	where the last inequality is from Lemma \ref{lem1a} with the condition $H\in(0,\frac12)$.
	At the same time, $J_9$ has the symmetry of $(k,l)$ and $(m,j)$. By a similar way we obtain
	\begin{equation}\label{63}
		\begin{split}
			\abs{J_9} \leq C n^H.
		\end{split}
	\end{equation}
	
	Group 4. Applying the symmetry of $m,j$ and Lemma \ref{lem1} to $J_{11}$:
	\begin{align}\label{64}
		\notag &\abs{J_{11}} = \abs{\sum_{j,k,l,m=1}^n xyz\rho(mh,jh)} \\
		\notag &\leq C  \sum_{j,k,l,m=1}^n \left( 1+(j \wedge k)  \right)^{2(H-1)} \cdot   \left( 1+\abs{k-l}  \right)^{H-1} \cdot \left( 1+\abs{l-m}  \right)^{H-1} \cdot \left( 1+\abs{m-j}  \right)^{2(H-1)}\\
		\notag &\leq C \Bigg[ \sum_{1 \leq m \leq k \leq j  \leq n,1 \leq l \leq n} \left( 1+k  \right)^{2(H-1)} \cdot   \left( 1+\abs{k-l}  \right)^{H-1} \cdot \left( 1+\abs{l-m}  \right)^{H-1} \cdot \left( 1+(j-m)  \right)^{2(H-1)}\\
		\notag &+ \sum_{1 \leq m \leq j \leq k  \leq n,1 \leq l \leq n} \left( 1+j  \right)^{2(H-1)} \cdot   \left( 1+\abs{k-l}  \right)^{H-1} \cdot \left( 1+\abs{l-m}  \right)^{H-1} \cdot \left( 1+(j-m)  \right)^{2(H-1)}\\
		\notag &+ \sum_{1 \leq k \leq m \leq j  \leq n,1 \leq l \leq n} \left( 1+k  \right)^{2(H-1)} \cdot   \left( 1+\abs{k-l}  \right)^{H-1} \cdot \left( 1+\abs{l-m}  \right)^{H-1} \cdot \left( 1+(j-m)  \right)^{2(H-1)} \Bigg]\\
		&\leq C n^{2H}, 
	\end{align}
	where  the last inequality is due to Lemma \ref{lem1a} with the condition $H\in(0,\frac12)$. Similarly, we have
	\begin{equation}\label{65}
		\begin{split}
			\abs{J_i} \leq C n^{2H}, \quad i=12,13,14.
		\end{split}
	\end{equation}
	
	Group 5. We using the symmetry of $j,k,l,m$ and Lemma \ref{lem1} to get that
	\begin{equation}\label{66}
		\begin{split}
			&\abs{J_{15}} = \abs{\sum_{j,k,l,m=1}^n xyzw} \\
			&\leq C  \sum_{j,k,l,m=1}^n \left( 1+(j \wedge k)  \right)^{2(H-1)} \cdot   \left( 1+(k \wedge l)  \right)^{2(H-1)} \cdot \left( 1+(l \wedge m)  \right)^{2(H-1)} \cdot \left( 1+\abs{m-j}  \right)^{H-1}\\
			&\leq C  \sum_{1 \leq j \leq k \leq l \leq m  \leq n} \left( 1+j  \right)^{2(H-1)} \cdot   \left( 1+k  \right)^{2(H-1)} \cdot \left( 1+l  \right)^{2(H-1)} \cdot \left( 1+(m-j) \right)^{H-1}\\
			&\leq C n^{H}, 
		\end{split}
	\end{equation}
	where  the last inequality is from Lemma \ref{lem1a} with the condition $H\in(0,\frac12)$.
	In conclusion, we obtain 
	\begin{equation}\label{67}
		\begin{split}
			J \leq \sum_{i=1}^{15}\abs{J_i} \leq Cn^{2H} \leq Cn^{\frac32}.
		\end{split}
	\end{equation}
	This completes the proof.
\end{proof}

\subsection{\bf Proofs of main Theorems}\label{sec.sub.4.2}

{\bf Proof of Theorem \ref{th2}.}
Following the proof methodology of Theorem \ref{th1}, we take the random variable $X$ as  the   second moment of sample about Ornstein-Uhlenbeck model $\{Z_t:t \geq 0\}$ defined as in \eqref{ou2} with the discrete form:
\begin{equation}\label{A_n}
	X=A_n:=\frac1n\sum_{j=1}^nZ_{jh}^2.
\end{equation}
Lemma \ref{lem2}  is also a key tool for proving Berry-Ess\'een upper bound of $\hat{\theta}_{n}$ defined in   \eqref{theta hat2}, which implies that there exists a positive constant $C$ independent of $n$ such that for $n$ large enough
\begin{equation}\label{36}
	\begin{split}
		&d_{Kol}\left(\sqrt{n}(\hat{\theta}_{n}-\theta), \mathcal{N}\right)\\
		&\le  C\times \left(d_{Kol}(\sqrt{n}(A_n-\E[A_n]), \varpi) +\sqrt{n}\abs{  \E[A_n] - a}+\frac{1}{\sqrt{n}}\right),
	\end{split}
\end{equation}
where $a=g(\theta)= H\Gamma(2H)\theta^{-2H}$ and  $\mathcal{N}, \,\varpi$  are the same as in estimation \eqref{30}. Throughout this proof, we assume  $H\in(0,\frac12)$.

To estimate $\sqrt{n}\abs{  \E[A_n] - a}$, we first note from \eqref{35} that it suffices to prove that $\sqrt{n}\abs{  \E(B_n - A_n)} \leq C\times \frac1{\sqrt{n}}$. In fact, the inequality \eqref{ap3} in Lemma \ref{lem1} implies that
\begin{equation}\label{37}
	\begin{split}
		\sqrt{n}\abs{  \E(B_n - A_n)} &\leq \frac1{\sqrt{n}} \sum_{j=1}^n \abs{\mathbb{E}\left( Z_{jh}^2- X_{jh}^2 \right)} \leq \frac{C}{\sqrt{n}} \sum_{j=1}^n \left( 1 \wedge (jh)^{2(H-1)} \right)\\
		&\leq \frac{C}{\sqrt{n}} \sum_{j=1}^n \left( 1 + jh \right)^{2(H-1)} \leq C\times n^{(2H-1) \vee 0 -\frac12}=C\times\frac1{\sqrt{n}}.
	\end{split}
\end{equation}

Next, using arguments analogous to those in the estimation of \eqref{34} combined with Propositions \ref{pro4}, \ref{pro5}, \ref{pro6}, we derive that  for $H\in(0,\frac12)$,
\begin{equation}\label{37a}
	\begin{split}
		d_{Kol}(\sqrt{n}(A_n-\E[A_n]), \varpi) \leq C 
		\times \frac{1}{\sqrt{n}} .
	\end{split}
\end{equation}
Substituting the estimations \eqref{37}, \eqref{37a} into \eqref{36} implies the final Berry-Ess\'een upper bound
\begin{equation}\label{37b}
	d_{Kol}\left(\sqrt{n}(\hat{\theta}_{n}-\theta), \mathcal{N}\right) \leq C_{\theta,H,h}\times    
	\frac{1}{\sqrt{n}}.
\end{equation}
\hfill$\Box$\medskip

{\bf Proof of Theorem~\ref{th2-2}.} 
The distinction between Theorem~\ref{th2} and Theorem~\ref{th2-2}  lies in improving Hypothesis~\ref{hyp1} with $H\in(0,\frac12)$  to Hypothesis~\ref{hyp1-1} with $H\in(\frac12,\frac34)$, thereby accommodating broader Gaussian processes such as the bi-fractional Brownian motion and the sub-bifractional Brownian motion. Notice that the estimations in Lemma \ref{lem1} play a  key rule in the proof of  Theorem~\ref{th2}. To establish Theorem~\ref{th2-2}, it suffices to  build up a new comparison about  covariance functions ${\rho}(t,s)$ and  $\tilde{\rho}(t,s)$, which is presented in the following Proposition.
\begin{proposition}\label{pro7}
	Let  ${\rho}(t,s)=\mathbb{E}(X_tX_s)$,  $\tilde{\rho}(t,s)=\mathbb{E}(Z_tZ_s)$ be the  covariance  function of the Ornstein-Uhlenbeck processes $X_t$ and $Z_t$ driven by fBm $B_t^H$ and $G_t$ satisfying Hypothese \ref{hyp1-1} with $H\in(\frac12,1)$. Then there exists  a constant $C\geq0$ independent of $T$ such that for any $0\leq s \leq t \leq T$,
	\begin{equation}\label{68}
		\abs{\tilde{\rho}(t,s) - \rho(t,s)} \leq C \left(1 \wedge s^{2(H-1)} \wedge (t-s)^{2(H-1)}\right).
	\end{equation}
	Moreover, the difference of variance of $X_t$ and $Z_t$ satisfies
	\begin{equation}\label{69}
		\abs{\mathbb{E}[Z_t^2] -\mathbb{E}[X_t^2]} \leq C \left(1 \wedge t^{2(H-1)}\right).
	\end{equation}
\end{proposition}
\begin{proof}
	For any $0\leq s \leq t \leq T$, according to the relationship \eqref{innp fg3-zhicheng0-0} between the inner products of two functions in the Hilbert spaces $\mathfrak{H}$ and $\mathfrak{H}_1$, we have 
	\begin{equation}\label{71}
		\abs{\mathbb{E}[Z_t^2] -\mathbb{E}[X_t^2]} \leq \int_0^te^{-\theta(t-u)} \dif u \int_0^te^{-\theta(t-v)} \abs{\frac{\partial^2R(u,v)}{\partial u \partial v} - \frac{\partial^2R^B(u,v)}{\partial u \partial v}} \dif v.
	\end{equation}
	At the same time,  Hypothesis~\ref{hyp1-1} with restriction $H\in(\frac12,1)$ implies that 
	\begin{equation}\label{73}
		\left| \frac{\partial}{\partial s}\left(\frac {\partial R(s,t)}{\partial t} - \frac {\partial R^{B}(s,t)}{\partial t}\right)\right| \le  C_1  (t+s)^{2H-2}+C_2 (s^{2H'}+t^{2H'})^{K-2}(st)^{2H'-1},
	\end{equation}
	where $H'\in (\frac12,1),\, K\in (0,2)$ and $H:=H'K\in (\frac12,1)$. Then, by the basic inequality $a+b\geq 2\sqrt{ab}$ and  Lemma \ref{lem1a1}, we obtain that
	\begin{equation}\label{74}
		\begin{split}
			\abs{\mathbb{E}[Z_t^2] -\mathbb{E}[X_t^2]} &\leq C_1 \int_0^te^{-\theta(t-u)} \dif u \int_0^te^{-\theta(t-v)} (u+v)^{2H-2} \dif v\\
			&+ C_2 \int_0^te^{-\theta(t-u)} \dif u \int_0^te^{-\theta(t-v)} (u^{2H'}+v^{2H'})^{K-2}(uv)^{2H'-1} \dif v\\
			&\leq C_1 \left[\int_0^te^{-\theta(t-u)} u^{H-1} \dif u \right]^2 + C_2 \left[\int_0^te^{-\theta(t-u)} u^{H'k-1} \dif u \right]^2\\
			&\leq C \left(1 \wedge t^{2(H-1)}\right),
		\end{split}
	\end{equation}
	which proves the inequality \eqref{71}. Furthermore, combining this result with the equations (3.37) and (3.41)  of \cite{CLZ24} yield 
	\begin{equation}\label{72}
		\begin{split}
			&\abs{\tilde{\rho}(t,s) - \rho(t,s)}\\
			&\leq C e^{-\theta (t-s)} \abs{\mathbb{E}[Z_s^2] -\mathbb{E}[X_s^2]} + \int_0^se^{-\theta(s-u)} \dif u \int_s^te^{-\theta(t-v)} \abs{\frac{\partial^2R(u,v)}{\partial u \partial v} - \frac{\partial^2R^B(u,v)}{\partial u \partial v}} \dif v\\
			&\leq C \left(1 \wedge s^{2(H-1)} \wedge (t-s)^{2(H-1)}\right) + \int_0^se^{-\theta(s-u)} \dif u \int_s^te^{-\theta(t-v)} \abs{\frac{\partial^2R(u,v)}{\partial u \partial v} - \frac{\partial^2R^B(u,v)}{\partial u \partial v}} \dif v
		\end{split}
	\end{equation} 
	where the last inequality is from the fact $ e^{-\theta (t-s)} \leq C \left(1 \wedge (t-s)^{2(H-1)}\right) $. 
	
	Next, we define a double integral as
	\begin{equation}\label{70}
		\begin{split}
			{\rm II} &:=\int_0^se^{-\theta(s-u)} \dif u \int_s^te^{-\theta(t-v)} \abs{\frac{\partial^2R(u,v)}{\partial u \partial v} - \frac{\partial^2R^B(u,v)}{\partial u \partial v}} \dif v\\
			&\leq C_1 \int_0^se^{-\theta(s-u)} \dif u \int_s^te^{-\theta(t-v)} (u+v)^{2H-2} \dif v\\
			&+ C_2 \int_0^se^{-\theta(s-u)} \dif u \int_s^te^{-\theta(t-v)} (u^{2H'}+v^{2H'})^{K-2}(uv)^{2H'-1} \dif v\\
			&:= {\rm II}_1 + {\rm II}_2,
		\end{split}
	\end{equation}
	under  Hypothesis~\ref{hyp1-1}. Let's estimate ${\rm II}$ in two parts based on the inequality \eqref{73}.\\
	Part 1. Making the change of variable $x=v-s$ implies that 
	\begin{equation}\label{75}
		\begin{split}
			{\rm II}_1 \leq C \int_0^se^{-\theta(s-u)} \dif u \int_0^{t-s}e^{-\theta((t-s)-x)} x^{2H-2} \dif x \leq C(t-s)^{2H-2} 
		\end{split}
	\end{equation}
	where in the last inequality we use Lemma \ref{lem1a1} with $H\in(\frac12,1)$. \\
	Part 2. Suppose that $H'\in (\frac12,1),\, K\in (0,2)$ and $H:=H'K\in (\frac12,1)$. We make the change of variable $x=v-s$ and use  the L'H\^{o}pital's rule to get that
	\begin{equation}\label{76}
		\begin{split}
			&\lim_{y\rightarrow \infty} \frac{{\rm II}_2}{y^{2(H'K-1)}}\\
			&\leq C \lim_{y\rightarrow \infty} \frac{\int_0^se^{-\theta(s-u)} u^{2H'-1} \dif u \int_0^{y}e^{\theta x} (u^{2H'}+(x+s)^{2H'})^{K-2}(x+s)^{2H'-1} \dif x}{y^{2(H'K-1)}e^{\theta y}}\\
			&=  C \lim_{y\rightarrow \infty} \frac{\int_0^se^{-\theta(s-u)} u^{2H'-1} \dif u  (u^{2H'}+(y+s)^{2H'})^{K-2}(y+s)^{2H'-1} }{\theta y^{2(H'K-1)} + 2(H'K-1)y^{2(H'K-1)-1}}\\
			&=  C \lim_{y\rightarrow \infty} \int_0^se^{-\theta(s-u)} (\frac{u}y)^{2H'-1}   \left[(\frac{u}y)^{2H'}+(1+\frac{s}y)^{2H'}\right]^{K-2} \dif u  \frac{(1+\frac{s}y)^{2H'-1} }{\theta  + 2(H'K-1)y^{-1}},
		\end{split}
	\end{equation}
	where the last equality is from $2(H'K-1)=2H'(K-2)+2(2H'-1)$. Notice that 
	\begin{equation}\label{77}
		\begin{split}
			\lim_{y\rightarrow \infty}   \frac{(1+\frac{s}y)^{2H'-1} }{\theta  + 2(H'K-1)y^{-1}} = \frac1{\theta}.
		\end{split}
	\end{equation}
	And then, because $u\in(0,s)$, $s$ is fixed, choosing $y>s$, we have 
	$$ (\frac{u}y)^{2H'-1}   \left[(\frac{u}y)^{2H'}+(1+\frac{s}y)^{2H'}\right]^{K-2} \leq 1.$$
	The Lebesgue dominated convergence theorem yields
	\begin{equation}\label{78}
		\begin{split}
			\lim_{y\rightarrow \infty} \int_0^se^{-\theta(s-u)} (\frac{u}y)^{2H'-1}   \left[(\frac{u}y)^{2H'}+(1+\frac{s}y)^{2H'}\right]^{K-2} \dif u  =0,
		\end{split}
	\end{equation}
	with the condition $H'\in (\frac12,1)$. Consequently, we have 
	\begin{equation}\label{79}
		\begin{split}
			\lim_{y\rightarrow \infty} \frac{{\rm II}_2}{y^{2(H'K-1)}} <+\infty,
		\end{split}
	\end{equation}
	which means that
	\begin{equation}\label{80}
		\begin{split}
			{\rm II}_2 \leq C(t-s)^{2(H'K-1)} = C (t-s)^{2(H-1)}.
		\end{split}
	\end{equation}
	In conclusion, we obtain that
	\begin{equation}\label{81}
		\abs{\tilde{\rho}(t,s) - \rho(t,s)} \leq C \left(1 \wedge s^{2(H-1)} \wedge (t-s)^{2(H-1)}\right).
	\end{equation}
\end{proof}

Based on the results  of Proposition \ref{pro7}, the conclusion of Theorem~\ref{th2-2}  can be readily verified via an approach parallel to that of Theorem~\ref{th2}. This completes the proof.
\hfill$\Box$\medskip

\section{Appendix}\label{sec appendix}
\setcounter{equation}{0}

We have been used the following technical inequalities repeatedly throughout the paper, which is cited from Chen et al.  \cite{CLSG,CDL24,CLZ24}.
	\begin{lemma}\label{lem1a1}
		Assume $\beta>-1$, $\theta>0$ and  two functions with form
		\begin{equation}\label{appendix1}
			\begin{split}
				A_1(t)=\int_0^te^{-\theta x}x^{\beta}\dif x, \qquad A_2(t)=\int_0^te^{-\theta (t-x)}x^{\beta}\dif x,
			\end{split}
		\end{equation} 
		then there exist a positive constant $C$ such that for any $s\in\left[0,\infty\right)$,
		\begin{equation}\label{appendix2}
			\begin{split}
				A_1(t)\leq C(t^{\beta+1}{\bf1}_{[0,1]}(t)+{\bf1}_{(1,\infty)}(t))\leq C(1\wedge t^{\beta+1}),
			\end{split}
		\end{equation}
		\begin{equation}\label{appendix3}
			\begin{split}
				A_2(t)\leq C(t^{\beta+1}{\bf1}_{[0,1]}(t)+t^{\beta}{\bf1}_{(1,\infty)}(t))\leq C(t^{\beta}\wedge t^{\beta+1}).
			\end{split}
		\end{equation}
		In particular, if $\beta\in(-1,0)$, then there exist a positive constant $C$ such that for any $s\in\left[0,\infty\right)$,
		\begin{equation}\label{appendix4}
			\begin{split}
				A_2(t)\leq C(1\wedge t^{\beta}).
			\end{split}
		\end{equation}
	\end{lemma}
	
	\begin{lemma}\label{lem1}
		Denote 
		$${\rho}(t,s)=\mathbb{E}(X_tX_s), \quad \tilde{\rho}(t,s)=\mathbb{E}(Z_tZ_s)$$
		by the  covariance  function of the Ornstein-Uhlenbeck processes $X_t$ and $Z_t$ driven by fBm $B_t^H$ and $G_t$ satisfying Hypothese \ref{hyp1}. Then there exists a positive constant $C$ independent of $T$ such that for any $0\leq s \leq t \leq T$,
		\begin{equation}\label{ap1}
			\abs{\rho(t,s)} \leq C \left(1 \wedge (t-s)^{2(H-1)}\right)\leq C \left(1+ (t-s)\right)^{2(H-1)},	
		\end{equation}
		\begin{equation}\label{ap2}
			\abs{\tilde{\rho}(t,s) - \rho(t,s)} \leq C \left(1 \wedge s^{2(H-1)} \wedge (t-s)^{H-1}\right).
		\end{equation}
		Moreover, the difference of variance of $X_t$ and $Z_t$ satisfies
		\begin{equation}\label{ap3}
			\abs{\mathbb{E}[Z_t^2] -\mathbb{E}[X_t^2]} \leq C \left(1 \wedge s^{2(H-1)}\right) \leq C \left(1+ s\right)^{2(H-1)}.
		\end{equation}
	\end{lemma}
	
	\begin{lemma}\label{lem1a}
		If $r\in\mathbb{N}:=\{1,2,\cdots\}$ is large enough and $v_1,\cdots,v_l$ are positive, then there exists a positive constant $C$ depending on $v_1,\cdots,v_l$ such that
		\begin{equation}\label{ap4}
			\sum_{r_i\in\mathbb{N},\sum_{i=1}^lr_i<r}r_1^{v_1-1}r_2^{v_2-1}\cdots r_l^{v_l-1} \leq C\times r^{\sum_{i=1}^lv_i}.
		\end{equation}
		At the same time, if $r\in\mathbb{N}:=\{1,2,\cdots\}$ is large enough and $v_1,\cdots,v_l$ are negative, then there exists a positive constant $C$ depending on $v_1,\cdots,v_l$ such that
		\begin{equation}\label{ap5}
			\sum_{r_i\in\mathbb{N},\sum_{i=1}^lr_i<r}r_1^{v_1-1}r_2^{v_2-1}\cdots r_l^{v_l-1} \leq C<\infty.
		\end{equation}
	\end{lemma}
	
	Next, we provide a detailed proof of the key Lemma~\ref{lem2} of this paper.\\
	{\bf Proof of Lemma~\ref{lem2}.} 
	Without loss of the generality, we assume $\sigma_1=1$ and generalize  the function  $f(x)=x^{-\frac{1}{\alpha}}, \, 0<\alpha<2$. Then we have $\xi\sim N(0,1)$ and  $\Phi(z)$ is its cumulative distribution function.  Thus, we have
	\begin{equation}\label{ap6}
		\begin{split}
			d_{Kol}(\sqrt{T}(f(X)-\theta) ,\, \xi)= \sup_{z\in \Rnum}\abs{P\left({\sqrt{T}}(f(X)-\theta)\le z\right)-\Phi( z)}			
		\end{split}
	\end{equation}
	Denote 
	\begin{equation}\label{ap7}
		\begin{split}
			A(z):=P\left({\sqrt{T}}(f(X)-\theta)\le z\right)-\Phi( z).
		\end{split}
	\end{equation}
	Since $X\ge 0$ almost surely and  $f(x)=x^{-\frac{1}{\alpha}}, \, 0<\alpha<2$, we have $f(X)> 0$ a.s. Hence, we shall assume $z>- \sqrt{T} \theta$. Otherwise, the standard estimate for a normal random variable $ \Phi(-t)\le \frac{1}{2t},\,\forall t>0$ yields 
	$$|A|=\Phi( z) \leq \Phi\big(-\sqrt{ T}\theta\big) \le\frac{C}{\sqrt{T}} \,.$$ 
	
	When $z>- \sqrt{T} \theta$, we have $\frac{z}{\sqrt{T}} + \theta>0$. Hence, the monotonicity of the function $f(x)=x^{-\frac{1}{\alpha}}$ implies that 
	$$\set{ X^{-\frac1{\alpha}} - \theta \leq \frac{z}{\sqrt{T}}}=\set{ X -\theta^{-\alpha} \geq \left( \frac{z}{\sqrt{T}} + \theta \right)^{-\alpha} - \theta^{-\alpha} }.$$
	Recall that $g(\theta)=\theta^{-\alpha}$. We have when $z>- \sqrt{T} \theta$, 
	\begin{equation*}\label{ap8}
		\begin{split}
			A(z)&=P\left(  X^{-\frac1{\alpha}} - \theta \leq \frac{z}{\sqrt{T}}  \right)-\Phi(z)\\
			&=P\left( X - \theta^{-\alpha} \geq \left( \frac{z}{\sqrt{T}} + \theta \right)^{-\alpha} - \theta^{-\alpha}   \right)-\Phi(z)\\
			&=P\left(\sqrt{\frac{T}{\sigma_2^2}} \left(X - g(\theta) \right) \geq \sqrt{\frac{T}{\sigma_2^2}}  \theta^{-\alpha} \left[ \left( \frac{z}{\sqrt{T}\theta} + 1 \right)^{-\alpha} - 1 \right] \right)-\Phi(z)\\
			&=P\left(\sqrt{\frac{T}{\sigma_2^2}} \left(X - g(\theta) \right) \geq \frac{\sqrt{T} \theta}{\alpha } \left[ \left( \frac{z}{\sqrt{T}\theta} + 1 \right)^{-\alpha} - 1 \right] \right)-\Phi(z), 		
		\end{split}
	\end{equation*}
	where in the last line, we use \eqref{fcha relat}  and $g'(\theta)=-\alpha\theta^{-\alpha-1}$. Next, we take the short-hand notion $\bar{\Phi}(z)=1-\Phi(z)$ and 
	$$\nu= \frac{\sqrt{T} \theta}{\alpha } \left[ \left( \frac{z}{\sqrt{T}\theta} + 1 \right)^{-\alpha} - 1 \right].$$
	It is clear that 
	\begin{align*}
		\set{\sqrt{\frac{T}{\sigma_2^2}} \left(X - g(\theta) \right) \geq \nu}=\set{\sqrt{\frac{T}{\sigma_2^2}} \left(X - \mathbb{E}[X] \right) \geq \nu - \sqrt{\frac{T}{\sigma_2^2}} \left( \mathbb{E}[X] - g(\theta) \right)}.
	\end{align*} Denote  $\bar{\Phi}(z)=1-\Phi(z)$.
	Hence,  the triangle inequality implies that
	\begin{equation}\label{ap9}
		\begin{split}
			\abs{A(z)}
			&\leq  \abs{P\left(\sqrt{\frac{T}{\sigma_2^2}} \left(X - \mathbb{E}[X] \right) \geq \nu - \sqrt{\frac{T}{\sigma_2^2}} \left( \mathbb{E}[X] - g(\theta) \right) \right) - \bar{\Phi} \left( \nu - \sqrt{\frac{T}{\sigma_2^2}} \left( \mathbb{E}[X] - g(\theta) \right) \right) }\\
			&\quad+ \abs{ \bar{\Phi} \left( \nu - \sqrt{\frac{T}{\sigma_2^2}} \left( \mathbb{E}[X] - g(\theta) \right) \right) - \bar{\Phi} \left( \nu  \right)  } + \abs{ \bar{\Phi} \left( \nu  \right) - \Phi(z) }\\
			&\le d_{Kol}(\sqrt{\frac{T}{\sigma_2^2}} \left(X - \mathbb{E}[X] \right),\,\xi )+\sqrt{\frac{T}{\sigma_2^2}} \abs{ \mathbb{E}[X] - g(\theta) }+ \abs{ \bar{\Phi} \left( \nu  \right) - \Phi(z) }\\
			&\leq C \Bigg( d_{Kol}(\sqrt{T}(X-\E[X]), \eta) +\sqrt{T}\abs{  \E[X] - g(\theta)}+\frac{1}{\sqrt{T}}\Bigg),
		\end{split}
	\end{equation}
	where in the last two inequalities we use  the standard estimate for the tail of a normal random variable, $\abs{\bar{\Phi}(z_1)-\bar{\Phi}(z_2) }\le \abs{z_1-z_2}$, and Lemma 5.4 of \cite{CZ21}. Subscribing the inequality  \eqref{ap9} into   \eqref{ap6} yields  \eqref{mubiao}.
	
	The following it taken from Lemma 5.4 of \cite{CZ21}.	
	
	\begin{lemma}\label{bound lem 54}
		Let $c>0$ be a constant. Denote $\nu(z)=  \frac{c}{ 2\beta}  \sqrt{T}  \Big[ \big( 1+\frac{z }{c \sqrt{T}}\big)^{-2\beta}-1 \Big], \, 0<\beta<1$, when $z> -c \sqrt{T}$ and $ \bar{\Phi}(z)=1-P(Z\le z)$. Then there exists some positive number $C$ independent of $T$ such that 
		\begin{align*}
			\sup_{z> -c \sqrt{T}}\abs{  \bar{\Phi}(\nu) -\Phi (z)}\le \frac{C}{\sqrt{T}}.
		\end{align*}
	\end{lemma}
	\begin{proof} We follow the line of the proof of Theorem 3.2  in \cite{SV 18}.
		By the mean value theorem, there exists some number $\eta(z)\in (0,1)$ such that 
		\begin{align*}
			\nu=- z \big( 1+\frac{z \eta}{c \sqrt{T}}\big)^{-2\beta-1} .
		\end{align*}Hence, 
		\begin{align*}
			\abs{  \bar{\Phi}(\nu) -\Phi (z)}&=\abs{  {\Phi}\Big( \big( 1+\frac{z \eta}{c \sqrt{T}}\big)^{-2\beta-1}\cdot z\Big) -\Phi (z)}\\
			&=\frac{1}{\sqrt{2\pi}}  {\int^{z}_{ z \big( 1+\frac{z \eta}{c \sqrt{T}}\big)^{-2\beta-1} } \,e^{-\frac{t^2}{2}}\dif t} \,.
		\end{align*}
		When $z\in (-c \sqrt{T}, -\frac12 c \sqrt{T}]$, it is obvious that
		\begin{align*}
			\frac{1}{\sqrt{2\pi}}  {\int_{ z \big( 1+\frac{z \eta}{c \sqrt{T}}\big)^{-2\beta-1}  }^{z} \,e^{-\frac{t^2}{2}}\dif t}\le \Phi(-\frac12 c \sqrt{T})\le \frac{C}{\sqrt{T}}.
		\end{align*}
		When $z\in( -\frac12 c \sqrt{T},\,0 ]$, we have
		\begin{align}\label{eqn. lem 54}
			\frac{1}{\sqrt{2\pi}}  {\int_{z \big( 1+\frac{z \eta}{c \sqrt{T}}\big)^{-2\beta-1}  }^{z} \,e^{-\frac{t^2}{2}}\dif t}&\le \abs{z} e^{-\frac{z^2}{2}}\Big(  \big( 1+\frac{z \eta}{c \sqrt{T}}\big)^{-2\beta-1} -1\Big) .
		\end{align}The mean value theorem implies that there exists some number $\eta'\in (0,1)$ such that
		\begin{align*}
			\big( 1+\frac{z \eta}{c \sqrt{T}}\big)^{-2\beta-1} -1= (-1-2\beta) \frac{z\eta}{c\sqrt{T}} \big( 1+\frac{z \eta\eta'}{c \sqrt{T}}\big)^{-2\beta-2}\le (1+2\beta) \left(\frac12\right)^{-2\beta-2}\frac{\abs{z}}{c\sqrt{T}}.
		\end{align*} 
		Substituting the above inequality into (\ref{eqn. lem 54}), and 
		since the function $f(z)= {z}^2 e^{-\frac{z^2}{2}}$ is uniformly bounded, we have that when $z\in( -\frac12 c \sqrt{T},\,0]$, 
		\begin{align*}
			\frac{1}{\sqrt{2\pi}}  {\int_{z \big( 1+\frac{z \eta}{c \sqrt{T}}\big)^{-2\beta-1}  }^{z} \,e^{-\frac{t^2}{2}}\dif t}&\le \frac{C}{\sqrt{T}}.
		\end{align*}
		Denote $a=\sup\set{z^2e^{-\frac{z^2}{2}}:\,z\in\Rnum}$ and denote the domain $$D:=\set{(s,z)\in \Rnum_{+}^2:\, \frac{1}{z} \big( 1+\frac{z \eta}{c \sqrt{T}}\big)^{-2\beta-1} \le s \le \frac{1}{z},\, z\in (0,\,c\sqrt{T})}.$$ We claim that the function
		\begin{align*}
			f_2(s, z) = z^2e^{-\frac{s^2 z^4}{2}},\quad (s, z)\in D
		\end{align*} is uniformly
		bounded in the domain $D$. In fact, for the above positive constant $a$, we have
		\begin{align*}
			f_2(s, z) =\frac{1}{s^2z^2} \left((sz^2)^2 e^{-\frac{(sz^2)^2}{2}}\right)\le \frac{a}{(sz)^2}\le a  \big( 1+\frac{z \eta}{c \sqrt{T}}\big)^{2(2\beta+1)} \le  2^{2(2\beta+1)}a.
		\end{align*}
		When $z\in (0,\,c\sqrt{T})$, using the mean value theorem  and making the change of variable $t = z^2 s$ together with the fact that 
		$f_2(s, z) $ is uniformly
		bounded in the above domain $D$, we conclude that  there exists a number $\eta'\in (0,1)$ such that 
		\begin{align*}
			{\int_{ z \big( 1+\frac{z \eta}{c \sqrt{T}}\big)^{-2\beta-1}  }^{z} \,e^{-\frac{t^2}{2}}\dif t} &={\int_{ \frac{1}{z} \big( 1+\frac{z \eta}{c \sqrt{T}}\big)^{-2\beta-1}  }^{\frac{1}{z}} \,\,z^2e^{-\frac{s^2z^4}{2}}\dif s}\\
			&\le C \frac{1}{z} \Big( 1-  \big( 1+\frac{z \eta}{c \sqrt{T}}\big)^{-2\beta-1}  \Big)\\
			&= C  (1+2\beta) \frac{1}{z} \big( 1+\frac{z \eta\eta'}{c \sqrt{T}}\big)^{-2\beta-2}  \frac{z \eta}{c \sqrt{T}} \\
			&\le \frac{C}{\sqrt{T}}.
		\end{align*}
		
		Finally, when $z\in [c\sqrt{T},\,\infty)$, we have that $${\Phi}(-z) \le {\Phi}(-c\sqrt{T})\le \frac{1}{2c\sqrt{T}}$$ and 
		$\nu(z)\le  \frac{c}{ 2\beta}  \sqrt{T}  \Big[ 2^{-2\beta}-1 \Big]$.
		Hence,
		\begin{align*}
			\Phi (\nu)&\le \Phi ( \frac{c}{ 2\beta}  \sqrt{T}  [2^{-2\beta}-1])\le \frac{\beta}{[1-2^{-2\beta}]c\sqrt{T}},\\
			\abs{  \bar{\Phi}(\nu) -\Phi (z)}&= \abs{  \bar{\Phi}(z) -\Phi (v)}\le  \bar{\Phi}(z) +\Phi (\nu)= {\Phi}(-z) +\Phi (\nu)\le \frac{C}{\sqrt{T}}.
		\end{align*}
	\end{proof}

\section*{Acknowledgments}
The authors wish to express their sincere thanks to the anonymous referees for their valuable suggestions and comments. 
 Y. Chen   is supported by the National Natural Science Foundation of China (No. 12461029);  Z. Tang are supported by  the Project of Basic Research of Science and Technology Plan of Gansu Province in 2025--Excellent doctoral program (25JRRA235) and  the ``Innovation Star" project of university graduate students in Gansu Province (2025CXZX-867); Y. Li  are supported by the Youth Project of Hunan Provincial Natural Science Foundation of China (No. 2024JJ6417), the Project of Education Department of Hunan Province (24B0187) and  the 111 Project (No. D23017); H. Yi  is supported by  the High level Talent Research Project of Liupanshui Normal University (LPSSYKYJJ202416).



\begin{thebibliography}{99}
	\bibitem{Alazemi2024} 
	F. Alazemi,  A. Alsenafi,  Y. Chen and  H. Zhou, Parameter Estimation for the Complex Fractional Ornstein-Uhlenbeck Processes with Hurst parameter $H\in (0,\,\frac12)$, {\it Chaos Solitons Fractals.} {\bf 188} (2024) 115556.
	
	
	\bibitem{BBES 23} M. Balde,  R. Belfadli and  K. Es-Sebaiy,   Kolmogorov bounds in the CLT of the LSE for Gaussian Ornstein Uhlenbeck processes, {\it Stoch. Dynam.} {\bf 23}(4) (2023) 2350029.
	
	
	
	
	
	
	
	\bibitem{BBNP 12}  H. Bierm\'e,  A. Bonami, I. Nourdin and G.  Peccati,   Optimal Berry-Ess\'een rates on the Wiener space: the barrier of third and fourth cumulants, {\it ALEA  Lat Am J Probab Math Stat.} {\bf 9}(2) (2012) 473--500.
	
	
	
	
	
	\bibitem{CDL24}
	Y. Chen,  Z. Ding and  Y. Li,  Berry-Ess\'een bounds and almost sure CLT for the quadratic
	variation of a class of Gaussian process, {\it Commun. Stat.Theory Methods.} {\bf 53} (2024) 3920-3939. 
	
	
	
	\bibitem{CGL24} 
	Y. Chen,  W. Gao and Y. Li, Statistical Estimations for Non-Ergodic Vasicek Model Driven by Gaussian Processes, (2024) \textit{preprint}. 
	
	\bibitem{chenkl2020} 
	Y. Chen,   N. Kuang and  Y. Li,  Berry-Ess\'een bound for the parameter estimation of fractional Ornstein-Uhlenbeck processes, {\it Stoch. Dynam.}
	{\bf 20}(4) (2020)  2050023.    
	
	\bibitem{CLSG}
	Y. Chen, Y. Li,  Y. Sheng and X.  Gu,   Parameter estimation for an ornstein-uhlenbeck
	process driven by a general gaussian noise with Hurst parameter
	$H\in(0,\frac12)$, {\it Acta Math. Sinica (Chinese Ser.)}  {\bf 43}(5) (2023) 1483--1518.
	
	
	
	
	
	
	
	
	
	\bibitem{CLZ24} 
	Y. Chen,  Y. Li and H.  Zhou,  Berry-Ess\'{e}en bounds for the statistical estimators of
	an Ornstein-Uhlenbeck process driven by a general
	Gaussian noise, (2024) \textit{preprint}. 
	
	
	\bibitem{CZ21} 
	Y. Chen and  H. Zhou, Parameter Estimation for an Ornstein-Uhlenbeck Processes driven by a general Gaussian Noise, {\it Acta Math. Sin. (Engl. Ser.)}  {\bf  41B}(2) (2021) 573--595.
	
	
	\bibitem{Cheridito} 
	P. Cheridito,  H. Kawaguchi,  and M. Maejima,  Fractional Ornstein-Uhlenbeck processes,
	{\it Electr. J. Prob.} {\bf 8} (2003) 1--14.
	
	
	\bibitem{DEKN 22}  
	S. Douissi,  K. E-Sebaiy,  G. Kerchev and I.  Nourdin,  Berry-Ess\'een bounds of second moment estimators for Gaussian processes observed at high frequency, {\it Electron. J. Stat.} {\bf 16}(1) (2022) 636--670.
	
	
	
	
	\bibitem{NoutyJourn2013} 
	C. El-Nouty and J. Journ\'e,  The sub-bifractional Brownian motion, {\it Studia Sci Math Hungar.} {\bf 50} (2013) 67--121.
	
	\bibitem{Khalifa} 
	K. Es-Sebaiy and F.    Alazemi,  New Kolmogorov bounds in the CLT for random ratios and applications,  {\it Chaos Solitons Fractals} {\bf  181} (2024) 114686.
	
	
	
	\bibitem{HH21}
	E.M. Haress and Y. Hu,  Estimation of all parameters in the fractional Ornstein–Uhlenbeck model under discrete observations, {\it Stat Inference Stoch Process} {\bf 24} (2021) 327--351.
	
	
	\bibitem{HR24}
	 E.M. Haress and A. Richard,  Estimation of several parameters in discretely-observed stochastic differential equations with additive fractional noise, {\it Stat Inference Stoch Process} {\bf 27} (2024) 641--691.
	
	
	
	
	
	
	
	
	
	
	
	\bibitem{Hu2019}
	Y. Hu,   D. Nualart and H.   Zhou,   Parameter estimation for fractional Ornstein-Uhlenbeck processes of general Hurst parameter, {\it Stat. Inference Stoch. Process.} {\bf 22} (2019)  111--142.
	
	
	\bibitem{Jolis 2007} 
	M. Jolis,  On the Wiener integral with respect to the fractional Brownian motion on an interval, {\it J. Math. Anal. Appl.} {\bf 330} (2007) 1115--1127 . 
	
	
	
	\bibitem{KMR17} 
	K. Kubilius,  Y. Mishura and K.   Ralchenko,  {\it Parameter estimation in fractional diffusion models},  Vol. 8. (Springer: Berlin, Germany, 2017).
	

	
	
	
	
	
	\bibitem{Mishura}
	YS. Mishura,  {\it Stochastic Calculus for Fractional Brownian Motion and Related Processes}, (Springer-Verlag: Berlin, Germany, 2008).
	
	\bibitem{Nourdin2012}
	I. Nourdin,  {\it Selected aspects of fractional Brownian motion}, Vol. 4. (Springer: Milan, Italy, 2012).
	
	
	
	
	\bibitem{NP2012}
	I. Nourdin and  G.  Peccati,  {\it Normal approximations with Malliavin calculus:
		from Stein’s method to universality}, Vol. 192. (Cambridge University Press: Cambridge, UK, 2012).
	
	
	\bibitem{NP2015}
	I. Nourdin and  G.   Peccati, The optimal fourth moment theorem, {\it Proc. Amer. Math. Soc.}  {\bf 143} (2015) 3123--3133.
	
	
	
	
	\bibitem{Sgh 2014} 
	A. Sghir,  The generalized sub-fractional Brownian motion, {\it Commun. Stoch. Anal.} {\bf 7} (2013) Article 2.
	
	
	
	
	
	
	\bibitem{SV 18}
	T. Sottinen and L.  Viitasaari,  Parameter estimation for the Langevin equation with stationary-increment Gaussian noise, {\it Stat. Inference Stoch Process.} {\bf 21}(3) (2018) 569--601.
	
	
	
	\bibitem{vand}
	A.W. van der Vaart,   {\it Asymptotic Statistics}, (Cambridge University Press: Cambridge, UK, 1998).
	
	
	
	
	
	\bibitem{Zili17}
	M. Zili,  Generalized fractional Brownian motion. {\it Mod. Stoch. Theory Appl.} {\bf 4}(1) (2017) 15--24.
	
	
	
\end{thebibliography}
\end{document}